\documentclass[preprint]{imsart}
\usepackage[noend]{algpseudocode}
\usepackage{subfigure}
\usepackage{graphicx}
\usepackage{amsthm,amsmath}
\usepackage[numbers,sort]{natbib}
\RequirePackage[colorlinks,citecolor=blue,urlcolor=blue]{hyperref}

\usepackage{amsfonts}
\usepackage{bbm}
\usepackage[skip=2pt,font=scriptsize]{caption}
\usepackage{subfigure}
\usepackage[english]{babel}
\usepackage{blindtext}
\usepackage{amsthm}
\usepackage{framed}
\usepackage{sectsty}
\usepackage{titlecaps}
\usepackage{hyperref}
\usepackage{amsmath,amssymb,amsfonts,bm}
\usepackage{amsthm,amsfonts}
\usepackage{mathrsfs}
\usepackage{xcolor}
\usepackage{float}
\usepackage{graphicx}
\usepackage[margin=1in]{geometry}
\usepackage{empheq}
\usepackage{mathtools}
\usepackage{algorithm}
\usepackage{algpseudocode}
\mathtoolsset{showonlyrefs}

\mathtoolsset{showonlyrefs}
\numberwithin{equation}{section}

\numberwithin{equation}{section}
\newtheorem{theorem}{Theorem}
\numberwithin{theorem}{subsection} 
\newtheorem{lemma}[theorem]{Lemma}

\newtheorem{coro}[theorem]{Corollary}
\theoremstyle{definition}
\newtheorem{defn}[theorem]{Definition}

\newtheorem{assump}{\textbf{Assumption}}
\theoremstyle{remark}
\newtheorem*{remark}{Remark}

\newcommand{\lemref}[1]{Lemma~\ref{#1}}

\newcommand{\Aaa}{\|A\|_{2\alpha}}
\newcommand{\Aaaa}{\| \tilde{A}\|_{2\alpha}}
\newcommand{\dbar}{\bar{d}}

\newcommand{\EEb}{\mathbb{E}_{\mu(\omega)}}

\newcommand{\abs}[1]{\lvert#1\rvert}

\begin{document}

\begin{frontmatter}

\title{Unbiased Sampling of Multidimensional Partial Differential Equations with Random Coefficients}
\runtitle{Unbiased Sampling of for random PDEs }
\runauthor{J. Blanchet, F. Li AND X. Li}
 \author{\fnms{Jose} \snm{Blanchet}\ead[label=c1]{jblanche@stanford.edu},\thanksref{t1}}
\thankstext{t1}{Support from NSF Grant DMS-1720451, DMS-1838576 and CMMI-1538217 is gratefully acknowledged by Jose Blanchet. }
\thankstext{t1}{Department of Management Science and Engineering, Stanford University, \printead{c1} } 
\author{\fnms{Fengpei} \snm{Li}\ead[label=d1]{fl2412@columbia.edu}\thanksref{t2}}
\thankstext{t2}{Department of Industrial Engineering \& Operations Research, Columbia University, \printead{d1} }
\and
\author{\fnms{Xiaoou} \snm{Li}\ead[label=e1]{lixx1766@ umn.edu}\thanksref{t3}}
\thankstext{t3}{Support from NSF Grant DMS-1712657 is gratefully acknowledged by Xiaoou Li.}
\thankstext{t3}{School of Statistics, University of Minnesota,  \printead{e1} }
\affiliation{ Columbia University,\thanksref{t1}\thanksref{t2} Stanford University\thanksref{t1} and University of Minnesota\thanksref{t3}}

\begin{abstract} [{ }]
	\quad An unbiased estimator with finite variance and computational cost for function values  of solutions from partial differential equations governed by random coefficients is constructed. We show the proposed estimator bypass the curse of dimensionality and can be applied in various disciplines. For the error analysis, we analyse the random partial differential equations by its connection with stochastic differential equations and rough path estimation. 
\end{abstract}


\end{frontmatter}

\section{Introduction}
\subsection{Motivation and background}
Consider the solution $u(x,t):\mathbb{R}^{d}\times \mathbb{R}^{+}\rightarrow 
\mathbb{R}$ of  the  following random parabolic partial differential equation(PDE): 
\begin{equation}
\begin{cases}
\partial_{t}u(x,t)=\boldsymbol{\mu}^{T} (x ) D_{x}u(x,t)+\frac{%
	1}{2} \cdot trace\left( \sigma \left( x\right) \sigma ^{T}\left( x\right)
D_{xx}u(x,t)\right) \\ 
u(x,0)=f(x)%
\end{cases}%
\label{eq:heat},
\end{equation}
where $\left\{ \boldsymbol{\mu} \left( x \right) :x\in \mathbb{R}^{d}\right\}
\subset \mathbb{R}^{d}$ is a random field in probability space $(\Omega,\mathcal{F},\mathbb{P})$ with its realization $\mu(\cdot,\omega):\mathbb{R}^d\rightarrow\mathbb{R}^d$ and $\left\{
\sigma (x)\sigma ^{T}(x):x\in \mathbb{R}^{d}\right\} \subset \mathbb{R}%
^{d\times d}$ is from some deterministic (non-random) function $\sigma(\cdot)$. On the other hand,  $D_{x}u(x,t)\in\mathbb{R}^d$ and $D_{xx}u(x,t)\in\mathbb{R}^{d\times d}$ in \eqref{eq:heat} denote the partial derivatives while the function $trace(\cdot)$ is the matrix trace operator. The function $%
\left\{ f\left( x\right) :x\in \mathbb{R}^{d}\right\} $ describes the initial condition.

The heat equation \eqref{eq:heat} is a classic PDE that with many
applications. In different cases,  the interpretations for the
coefficients and the solution $u(x,t)$  are different. For example, in the theory of thermal
conductivity, the heat equation \eqref{eq:heat} follows from Fourier's law and the
solution $u(x,t)$ represents the temperature of the material at the location $x$ and time $t$%
, while the coefficients $\boldsymbol{\mu} $ and $\sigma $ characterize the thermal
conductivity of the material. On the other hand, in the theory of flow dynamics \cite{whitaker1986flow}, the heat equation \eqref{eq:heat} follows from Darcy's law when one tries to
describe the flow of fluids through a porous medium. Here, the solution $u(x,t)$ represents the fluid
pressure at location $x$ and time $t$, while the coefficients $\boldsymbol{\mu}$ and $\sigma$ characterize the
medium permeability. In both  cases, the coefficients at 
location $x$ reflect medium property, which are modeled as
random fields due to the heterogeneity of the media in practical applications. A partial literature on the modeling and analysis for the
heterogeneous random medium includes \cite%
{de2005dealing,ostoja2007microstructural,sobczyk2012stochastic}. However, for applications in finance, the equation (\ref{eq:heat}) could represent the price of a
European contract with payoff given by $f(x) $ at maturity
\cite{duffie2010dynamic}. In this case, the introduction of a random $\boldsymbol{\mu} $ and a deterministic 
$\sigma  $ is justified by the fact that the diffusion coefficient ${\sigma}$ can
often be estimated with reasonable accuracy in the setting of financial applications due to the characteristics of quardratic variations while the drift coefficient $\boldsymbol{\mu}$ is
typically difficult to calibrate \cite{hofmann1999lp,pastorello1996diffusion}. Thus, in this paper, we focus our attention on the cases of $\sigma(\cdot)$ taken to be determinstic whereas we note the proof can also be extended to incorporate sutiable assumptions on the randomness of $\sigma (\cdot)$. 

In these applications, since $\boldsymbol{\mu}$ is a random, the
solution $u(x,t): \mathbb{R}^d\times\mathbb{R}^{+}\rightarrow \mathbb{R}$ to Equation \eqref{eq:heat}  is also random. We use $\boldsymbol{u}(x,t)$ to denote the solution of \eqref{eq:heat} when the field  $\boldsymbol{\mu}$ is random  and use $u(x,t)$ when the field $\mu$ (or $\mu(\cdot,\omega)$) is fixed. As it turned out, in the context of random PDEs, it is common to evaluate expectations of the form
\begin{equation}
\nu =\mathbb{E}\left[ G(\boldsymbol{u}(x_{1},t_{1}),...,\boldsymbol{u}(x_{k},t_{k}))\right] ,
\label{Exp_G}
\end{equation}%
for certain values of $x_i,t_i, 1\leq i \leq k$ and some given function $G:\mathbb{R}^{d}\rightarrow \mathbb{R}$. As we shall see later, this task presents a analytic challenge
and it is natural for one to use Monte Carlo. In this paper, we introduce a methodology that provides an unbiased estimator for $\nu$ in \eqref{Exp_G} and  could be easily implemented by parallel
computing architectures.
\subsection{Main contribution}

Under
reasonable regularity conditions to be specified in Theorem~\ref{thm:main}, we construct a random
variable $W$ satisfying unbiasedness with $\mathbb{E}\left( W\right) =\nu$, finite variance with ${V}ar\left( W\right) <\infty $ and finite expected cost, i.e., the computational cost to simulate $W$ has finite expectation. Consequently, one can then generate $n$
independent copies of $W$ in parallel servers
and combine them to provide an estimate as well as 
confidence intervals for $\nu$ in \eqref{Exp_G}  with O$\left( n^{-1/2}\right) $ rate of convergence 
dictated by the central limit theorem (CLT). Specifically, if the parallel computing cores are relatively cheap and
wall-clock time is a relatively hard constraint, then the estimator $W$ we propose in
this paper is precisely the type of solution one wants to use in order to estimate $\nu$ in \eqref{Exp_G}.

As far as we know, our paper is the first to introcude unbiased estimators of $\nu$
with square-root convergence rate for  arbitrary dimension $d$ in the PDE \eqref{eq:heat}. For example,  the unbiased estimator proposed in \cite%
{li2016multilevel} is related to the solution of elliptic equations with random
inputs and Dirichlet boundary conditions. However, even though the sampling strategy in \cite{li2016multilevel} achieves square-root convergence rate, the estimator has finite variance only if $d\leq 3$. In other words, the procedure in \cite{li2016multilevel} suffers from the curse of dimensionality. In particular, the procedure in \cite{li2016multilevel} is to numerically solve the PDE using the finite element method (FEM)
whose error analysis on the rate of convergence depends 
on the underlying dimension $d$.  In fact, there has been a substantial amount of recent literature combining the multilevel Monte Carlo technique with the
numerical methods for PDE, all of which suffers from the curse of dimensionality, as the rate of convergence deteriorates with the increase of
problem dimensions \cite
{li2016multilevel, charrier2013finite, cliffe2011multilevel, mishra2012multi}. On the other hand, other available methods in the literature such as \cite%
{teckentrup2015multilevel,giles2015multilevel,crevillen2017multilevel} produce biased
estimators. 
In contrast, our method allows for a full Monte Carlo procedure with a
traditional square-root convergence rate for any dimension $d$. Thus, our proposed method preserves the
well-known characteristic of the Monte Carlo method in effectively combating the curse of dimensionality.
Although the constants in the 
convergence rate analysis of Monte Carlo depend on the dimension $d$, the
convergence rate as a function of the total number of random variables
generated (the level of simulation accuracy) is of the same order for any $d$. 

\subsection{Technical  contribution}

In this paper, we also exploit the connection between the parabolic PDE and stochastic
differential equations (SDE) in order to construct $%
W$. In particular, given a realization of the random field $\mu \left( \cdot ,\omega \right) $, it follows from the celebrated Feynman-Kac formula that one can represent the solution of Equation \eqref{eq:heat} $u\left( x,t\right) $ as the
expectation involving the solution of SDEs. Thus, conditioning on $\mu \left( \cdot ,\omega \right) $, one can use 
a multilevel Monte Carlo construction in \cite{giles2014}  that efficiently discretizes the underlying SDE  to reduce variance and combine the method with a
randomization step in \cite{rhee2015} to construct an unbiased estimator. Finally, we introduce an additional randomization technique similar to that in \cite%
{blanchet2015unbiased}  to account for the randomness of $\boldsymbol{\mu} $. 

On the other hand, in terms of technical contribution, the error analysis of the additional randomization
step requires a non-standard technical development. In particular,
conditioning on $\mu \left( \cdot ,\omega \right) $, the standard 
error analysis in \cite{giles2014,kloeden2011numerical}
would yield a term with infinite expectation,  preventing us from showing the finite variance property of our estimator. This is due to the
presence of the famous Gronwall's inequality \cite{howard1998gronwall} as a common tool in stochastic analyses (see \cite{kloeden2011numerical} and the remark
following Lemma~\ref{orderfirstlemma}). 

In order to overcome this issue, we use the theory of rough paths to
create path-by-path estimates. The
theory of rough paths \cite{lyons1998differential,davie2008differential,friz2010multidimensional,friz2014course} has received substantial attention in the literature due
to its connection to the theory of regularity structures and its
implications in nonlinear stochastic PDEs  \cite%
{hairer2014theory}. On the other hand, a significant
amount of literature has also been devoted to the connection between the theory of rough paths and stochastic numerical analysis in the setting of cubature methods \cite%
{lyons2004cubature} or SDEs \cite%
{bayer2016rough,blanchet2017}. In this light, our paper is the
first to connect rough paths estimates with the numerical analysis of random PDEs and thus adds to the growing literature combining the theory of
rough paths with numerical stochastic analysis. 
\section{Main results  }\label{sec:main}
\subsection{Assumptions and technical conditions}
\begin{assump}
	\label{assump:1} The random field $\boldsymbol{\mu}(\cdot):\mathbb{R}^d\rightarrow\mathbb{R}^d$ has the following expansion 
	\begin{equation}\label{field}
	\boldsymbol{\mu} (\cdot )=\sum_{i=1}^{\infty }\frac{\lambda _{i}}{i^{q}}\cdot\boldsymbol{V}_{i}\cdot\psi _{i}(\cdot)
	\end{equation}%
	where $q>4$ is a fixed constant, $\{\lambda_i\}_{i \geq 1}$ is a uniformly
	bounded sequence and $\{\boldsymbol{V}_{i}\}_{i \geq 1}$ independent $d$ dimensional Gaussian vectors  $\mathcal{N}(\boldsymbol{0},\Sigma_i)$, with  $\|\Sigma_i\|_{F}<L$,
	for all $i \geq 1$  and a constant $L>1$,  with $\|\cdot\|_{F}$ denoting the Frobenius norm. Moreover, $\psi _{i}(\cdot):\mathbb{R}%
	^{d}\rightarrow \mathbb{R}$ ($i=1,2,...$) is a sequence of deterministic functions that, for all $0\leq k,l\leq d$,
	\begin{equation}
	\| \psi _{i}\|_{\infty} < L,~~\| \frac{\partial {\psi _{i}}}{%
		\partial x_{l}}\|_{\infty} < iL,\text{ and }~~\| \frac{\partial ^{2}%
		{\psi _{i}}}{\partial x_{k}\partial x_{l}}\|_{\infty} < i^{2}L,
	\end{equation}
	for a constant $L>1$ with $\|\cdot\|_{\infty}$ denoting the supremum norm for functions.
\end{assump}
\begin{remark}
	The requirement on $\boldsymbol{V}_i$ can be relaxed by requiring that the tails of $\|\boldsymbol{V}_i\|_{\infty}$ decay faster than exponential functions uniformly in $i$. We focus on the Gaussian case for concreteness.
\end{remark}
\begin{defn}
	Denote $\mathcal{L}_1$ to be the space of  bounded, Lipschitz continuous and
	twice continuously differentiable fields where each of its element $\mu(\cdot):\mathbb{R}^d\rightarrow\mathbb{R}^d$ satisfies,
	\begin{equation}
	\|\mu\|_{\infty}< L_1, ~~ \| 
	\frac{\partial \mu_i}{\partial x_l}\|_{\infty}<L_1 ~~\text{ and }~
	\| \frac{\partial^2 \mu_{i}}{\partial x_k\partial x_l}%
	\|_{\infty} <L_1
	\end{equation}
	for  $1\leq i,k,l\leq d$ and some positive $L_1<\infty$ depending on $\mu(\cdot)$. Then, we define $L_1$ to be a bounding number for $\mu(\cdot)$.
\end{defn}

\begin{lemma}\label{assump1}
	\label{lemma L1} Under Assumption~\ref{assump:1}, $\boldsymbol{\mu}(\cdot) \in \mathcal{L}_1$ almost surely and for $n\geq 0$, the partial sum $\boldsymbol{S}_n=\sum_{i=1}^{n }\frac{\lambda _{i}}{i^{q}}\cdot\boldsymbol{V}_{i}\cdot\psi _{i}(\cdot)\in\mathcal{L}_1$ almost surely. Furthermore, there exists a random variable $\boldsymbol{L}_1>1$ 	with $\mathbb{E}(e^{t\boldsymbol{L}_1}) < \infty$ for all
	$t\in\mathbb{R}$ (i.e., well-defined moment-generating function) and it is a bounding number for $\boldsymbol{\mu}$ and $\{ \boldsymbol{S}_n \}_{n \geq 0}$ 
	almost surely.
\end{lemma}
\begin{proof}
	See Section~\ref{sec:proof}.
\end{proof}

\begin{assump}
	\label{assump:2}  There exists a constant $L>1$
	such that for $1\leq i,j,l, i',j'\leq d^{\prime}$, 
	\begin{equation}
	\| \sigma _{i'}\|_{\infty} < L,~~\| \frac{\partial {\sigma }%
		_{i'j'}}{\partial x_{l}}\|_{\infty} < L~~\text{and}~~ \|\frac{ \partial ^{2}{\sigma }%
		_{i'j'}}{\partial x_{k}\partial x_{l}}\|_{\infty} < L,
	\end{equation}
	\begin{equation}
	\|\frac{ \partial f}{\partial x_{i}}\|_{\infty} < L ~~\text{and}~~ \| \frac{\partial
		^{2}f}{\partial x_{i}\partial x_{j}}\|_{\infty} < L.
	\end{equation}
\end{assump}

\begin{assump}
	\label{assump:3} There exists a positive constant $%
	1<L<\infty$ such that for  $1\leq i,j\leq k$,
	\begin{equation*}
	\|\frac{\partial ^{2}G}{\partial
		x_{i}\partial x_{j}}\|_{\infty}< L~~\text{ and }~~\|\frac{\partial G}{\partial x_{i}}\|_{\infty}< L.
	\end{equation*}
\end{assump}
\subsection{Main theorem}
\begin{theorem}
	\label{add5} \label{thm:main} Under Assumptions~\ref{assump:1}-\ref{assump:3}%
	, we can construct a random variable $W$,  which is an unbiased estimator for $\nu$ defined
	in \eqref{Exp_G}. 
	Moreover, $W$ has a finite variance and  the computational cost for simulating one copy of $W$ has a finite expectation.

\end{theorem}
\begin{proof}
	See Section \ref{sec:proof-main}.
\end{proof}

\section{Construction of the estimator $W$}

\subsection{Preliminaries: Probabilistic representation of the solution  $u(x,t)$ }
Fixing $\mu\in\mathcal{L}_1$, the solution $u(x,t)$  to the PDE in %
\eqref{eq:heat} and certain $d$-dimensional diffusion process
are connected by the Feynman-Kac formula \cite{karatzas2012brownian}.
\begin{theorem}[Feynman-Kac Formula]\label{fkt}
	Fix $x\in\mathbb{R}^d$, $t \in \mathbb{R}^{+}$ and functions  $\sigma(\cdot),f(\cdot)$ satisfying Assumption~\ref{assump:2}.  For any $\mu(\cdot) \in \mathcal{L}_1$, the solution $u(x,t)$ of \eqref{eq:heat} satisfies
	\begin{equation}
	u(x,t)=\mathbb{E}[f(X_t)],
	\end{equation}
	where the expectation is taken w.r.t. to the $d$-dimensional diffusion process $X_s$ with $X_0=x$ following SDE (i.e., the unique strong solution),
	\begin{equation}\label{eq:SDE}
	dX_{s} =\mu (X_{s} )dt+\sigma (X_{s})dB_{s}%
	\quad\text{ for }\quad s>0,
	\end{equation}%
	with $B_s$ is $d'$-dimensional Brownian motion.
\end{theorem}
\begin{proof}
	Fix $\mu\in\mathcal{L}_1$, the existence and uniqueness of strong solution $\{X_s, s\geq 0\}$  are guaranteed by Assumption~\ref{assump:2} on the Lipschitz continuity of $\sigma(\cdot)$. The rest follows from the Feynman-Kac formula. For more details on the proof, See Chapter 4.4 in \cite{karatzas2012brownian}.
\end{proof}
Thus, fixing any  $\mu(\cdot)\in\mathcal{L}_1$, if we could build estimators $Z_i(\mu)$ satisfying 
$\mathbb{E}[Z_{i}(\mu)]=g(x_i,t_i;\mu)$,
and estimator $W(\mu)$ satisfying 
$\mathbb{E}[W(\mu)]=G(\mathbb{E}[Z_{1}(\mu)],...,%
\mathbb{E}[Z_k(\mu)])$, then based on Theorem~\ref{fkt}, formally we would have
\begin{equation}\label{intuition}
\mathbb{E}_{\mu\sim\boldsymbol{\mu}}[\mathbb{E}[W(\mu)]]=G(\mathbb E[Z_1(\boldsymbol{\mu})],...,\mathbb E[Z_k(\boldsymbol{\mu})])=\mathbb{E}[G(\boldsymbol{u}(x_{1},t_{1}),...,\boldsymbol{u}(x_{k},t_{k}))]=\nu,
\end{equation}
which is the desired property of $W$. We also note that the derivations in \eqref{intuition} assumes the measurability of $G(\boldsymbol{u}(x_{1},t_{1}),...,\boldsymbol{u}(x_{k},t_{k}))$(i.e., the well-definedness of $\nu$ ). This result can be proven using standard techniques and we omit it here.
\subsection{Step 1: Unbiased estimator $Z_i(\mu)$ for $g(x_i,t_i;\mu)$} 

\subsubsection{Variance Reduction and Bias Removal.}

Following the discussion above, given  $x,t$ and $\mu\in\mathcal{L}_1$, we first want to construct an estimator $Z(\mu)$ with $\mathbb{E}[Z(\mu)]=\mathbb{E}[f(X_t)].$ To estimate  $\mathbb{E}[f(X_t)]$, one way is to solve for the SDE \eqref{eq:SDE} numerically (e.g., Euler scheme, Milstein scheme \cite{kloeden2011numerical}) and use numerical solution $f(\hat X_t)$ as ``plug-in" estimators. However, such estimators have bias and suboptimal computational cost versus variance ratio as addressed in \cite{giles2013multilevel}. Before we remove bias, we first reduce the variance by multilevel Monte Carlo method  \cite{li2016multilevel,bayer2016rough,teckentrup2015multilevel,giles2013multilevel, giles2015multilevel}. For convenience, from now on, we assume $t=1$ and estimate $\mathbb{E}[f(X_1)]$. Now, given   $X_n(1)$, a numerical solution of the SDE at $t=1$ based on a level $n$ discretization, and a random variable $\Delta_n$ satisfying
\begin{equation}  
\mathbb{E}\Delta_n=
\mathbb{E}f(X_{n+1}(1))-\mathbb{E}f(X_n(1)).\label{eq:deltal}
\end{equation}
Then, a multilevel Monte Carlo(MLMC) estimator $Z_{\text{MLMC}}$ for $\mathbb{E}[f(X_1)]$ takes the form

\begin{equation}  \label{eq:MLMC}
Z_{\text{MLMC}} =\sum_{n=0}^N \frac{1}{M_n}\sum_{i=1}^{M_n}\Delta_n^{(i)} + \frac{1}{N_0}\sum_{i=1}^{N_0}f(X_0^{(i)}(1)),
\end{equation}
where  $\{
\Delta_n^{(i)}\}_{1\leq i \leq M_n}$ and $\{X_0^{(i)}(1)\}_{1\leq i \leq M_0}$ are I.I.D. copies. Usually, $N$ is
a large positive integer so the bias $|\mathbb{E} Z_{\text{MLMC%
}}-\mathbb{E}f(X_1)|= |\mathbb{E}f(X_{N+1}(1))-\mathbb{E}f(X_1)|$,
is small.  The optimal choice of $M_n$ depends on the variance and computational cost of $\Delta_n$, and  the  construction of $\Delta_n$'s satisfying \eqref{eq:deltal} usually involves variance reduction, an important technique in MLMC \cite{giles2008multilevel,giles2013multilevel}. In the next subsection, we further introduce a form of antithetic construction of $\Delta_n$, which reduces
$\mathbb{E} \Delta_n^2$ to order  $O(\Delta t_n^{2})$.

\subsubsection{Antithetic Multilevel Monte Carlo.}
One original kind of antithetic multilevel Monte Carlo construction of $\Delta_n$ is proposed in \cite{giles2014} for multidimensional SDEs. In this paper, we make modifications to the antithetic scheme in \cite{giles2014} to approximate the random field $\boldsymbol{\mu}$. Specifically, define
\begin{equation*}
\boldsymbol{\mu}^{(n)}(\cdot) \triangleq \sum_{i=1}^{\lfloor 2^{n\gamma} \rfloor} 
\frac{\lambda_i}{i^{q}}\cdot \boldsymbol{V}_i\cdot\psi_i(\cdot).
\end{equation*}
It follows from  Lemma~\ref{assump1} that $\boldsymbol{\mu}^{(n)}\in\mathcal{L}_1$ with the same bounding number $\boldsymbol{L}_1$ of $\boldsymbol{\mu}$. 
\begin{defn}\label{cdefn}
	Denote $\Delta
	t_{n}\triangleq 2^{-n}$ and  $%
	t_{k}^{n}\triangleq k\Delta t_{n}$. Define
	$\Delta B_{k}^{n}\triangleq
	B(t_{k+1}^{n})-B(t_{k}^{n})$ and $\Delta B_{j,k}^{n}\triangleq
	B_{j}(t_{k+1}^{n})-B_{j}(t_{k}^{n})$,
	to be the Brownian increments and its $j$th dimension component from Brownian motion. Finally, for $1\leq i, j\leq d^{\prime}$ and $i\neq j$,
	define
	\begin{equation*}
	\widetilde{A}_{i,j}(s,t)\triangleq \frac{
		(B_{i}(t)-B_{i}(s))(B_{j}(t)-B_{j}(s))}{2}
	\quad \text{and}\quad \widetilde{A}_{i,i}(s,t)\triangleq \frac{(B_{i}(t)-B_{i}(s))^{2}-(t-s)}{2},
	\end{equation*}%
\end{defn}
In particular, we approximate $\boldsymbol{\mu}$ by $\boldsymbol{\mu}^{(n)}$ and the process $\int_{s}^{t}\left(
B_{i}(r)-B_{j}(s)\right) dB_{j}(r)$ by $\widetilde{A}_{i,j}$.
Then $X_{i,n}(\cdot )$, the $i$th component $X_{n}(\cdot )$, is defined by the recursion
\begin{align}\label{coarse}
X_{i,n}(t_{k+1}^{n})=& X_{i,n}(t_{k}^{n})+\mu
_{i}^{(n)}(X_{n}(t_{k}^{n}))\Delta t_{n}+\sum_{j=1}^{d^{\prime}}\sigma
_{ij}(X_{n}(t_{k}^{n}))\Delta B_{j,k}^{n}  \notag   \\
& +\sum_{j=1}^{d^{\prime}}\sum_{l=1}^{d}\sum_{m=1}^{d^{\prime}}\frac{\partial \sigma _{ij}}{%
	\partial x_{l}}(X_{n}(t_{k}^{n}))\sigma _{lm}(X_{n}(t_{k}^{n}))\widetilde{A}%
_{mj}(t_{k}^{n},t_{k+1}^{n}),
\end{align}
for $0\leq k\leq 2^{n}-1$. We summarize the procedure in Algorithm \ref{Num} for $\textbf {``Num\_Sol"}$.

\makeatletter
\def\BState{\State\hskip-\ALG@thistlm}
\makeatother

\begin{algorithm}[H]
	\caption{Generate numerical solutions of the SDE}
	\begin{algorithmic}[1]
		\Procedure{Num\_ Sol}{$x, n, \{ \Delta B_k^n \}_{0\leq k \leq 2^n-1}$} with starting point $x \in \mathbb R^d$, level $n \geq 0$, $2^n$ Brownian increments $ \{ \Delta B_k^n \}_{0\leq k \leq 2^n-1}$ from the Brownian path.\label{Num}
		\State $X_{n}(0) \gets x \text{ and } {\mu}^{(n)}(\cdot) \gets \sum_{i=1}^{\lfloor 2^{n \gamma} \rfloor}\frac{\lambda_i}{i^q} \boldsymbol{V}_i  \phi_i(\cdot).$
		\For{$ 0\leq k \leq 2^n-1$}
		\For {$1\leq i \leq d'$}
		\State {$\widetilde{A}_{i,i}(t^n_k,t^n_{k+1})\gets \frac{(\Delta B^n_k)^{2}-\Delta t_n}{2},$}
		\For {$1 \leq j \leq d'$ and $j \neq i$}
		\State {$\widetilde{A}_{i,j}(t^n_k,t^n_{k+1})\gets \frac{(\Delta B^n_{i,k})(\Delta B^n_{j,k})}{2}.$}
		\EndFor
		\EndFor 
		\EndFor
		\For { $1 \leq k \leq 2^n-1$ and $1 \leq i \leq d$}
		\State {$X_{i,n}(t_{k+1}^{n})\gets X_{i,n}(t_{k}^{n})+\mu
			_{i}^{(n)}(X_{n}(t_{k}^{n}))\Delta t_{n}+\sum_{j=1}^{d^{\prime}}\sigma
			_{ij}(X_{n}(t_{k}^{n}))\Delta B_{j,k}^{n}$} \State{$\qquad\qquad\qquad+\sum_{j=1}^{d^{\prime}}\sum_{l=1}^{d}\sum_{m=1}^{d^{\prime}}\frac{\partial \sigma _{ij}}{%
				\partial x_{l}}(X_{n}(t_{k}^{n}))\sigma _{lm}(X_{n}(t_{k}^{n}))\widetilde{A}%
			_{mj}(t_{k}^{n},t_{k+1}^{n})$}
		\EndFor
		\EndProcedure
	\end{algorithmic}
\end{algorithm}

\begin{defn}\label{adefn}
	Fixing $n \geq 0$, given a sequence of Brownian increments $\{ \Delta B_k^n \}_{0 \leq k \leq 2^n-1 }$, we define the sequence of antithetic Brownian increments $\{\Delta B^{n,a}_k\}_{0\leq k \leq 2^{n}-1}$ by
	\begin{equation}\label{antitheticB}
	\Delta {B}_{2m}^{n,a} \triangleq \Delta B_{2m+1}^{n} \quad \text{and}\quad
	\Delta {B}_{2m+1}^{n,a}  \triangleq \Delta B_{2m}^{n} \qquad \text{for $0\leq m \leq 2^{n-1}-1$ }.
	\end{equation}
\end{defn}
The antithetic approximation $X_{n+1}^a(\cdot)$ follows the recursion:
\begin{align}
X_{i,n+1}^{a}(t_{k+1}^{n+1})& =X_{i,n+1}^{a}(t_{k}^{n+1})+\mu
_{i}^{(n+1)}(X_{n+1}^{a}(t_{k}^{n+1})){\Delta t_{n+1}}+\sum_{j=1}^{d^{\prime}}\sigma
_{ij}(X_{n+1}^{a}(t_{k}^{n+1}))\Delta {B}_{j,k}^{a,(n+1)}  \notag
\label{antithetic1} \\
& \qquad +\sum_{j=1}^{d^{\prime}}\sum_{l=1}^{d}\sum_{m=1}^{d^{\prime}}\frac{\partial \sigma
	_{ij}}{\partial x_{l}}(X_{n+1}^{a}(t_{k}^{n+1}))\sigma
_{lm}(X_{n+1}^{a}(t_{k}^{n+1}))\widetilde{A}%
_{mj}^{a}(t_{k}^{n+1},t_{k+1}^{n+1}).
\end{align}
In other words, $X_{n+1}^a (\cdot)\gets \textbf{Num\_Sol}(x,n+1,\{\Delta B^{n+1,a}_k \}_{1 \leq k \leq 2^{n+1}-1})$. However, from Definition \ref{adefn}, we have the important relations that \begin{align}  \label{trihere}
\Delta B^{n+1,a}_{2k}+ \Delta B^{n+1,a}_{2k+1} =\Delta B^{n+1}_{2k+1}+ \Delta B^{n+1}_{2k} =\Delta B^n_k=B(t_{k+1}^n)-B(t_{k}^n)
\end{align}
by summing the equations in \eqref{antitheticB}. This suggests the Brownian increments and the antithetic increments at time step $\Delta t_{n+1}$ produces the same increments at time step $\Delta t_n$. This motivates the construction of $\Delta_n$ as in \cite{giles2014} with reduced variance:
\begin{equation}  \label{eq:Delta_given_b}
\Delta_n\triangleq\frac{1}{2}\big( f(X_{n+1}^{f}(1))+f(X_{n+1}^{a}(1)) \big)%
-f(X_n(1)).
\end{equation}

\begin{algorithm}[H]
	\caption{Generate $\Delta_n$}
	\begin{algorithmic}[1]
		\Procedure{Delta\_Gen}{$x, n$} with input starting point $x \in \mathbb R^d$ and level $n \geq 0$.\label{Delta_Gen}
		\For {$0 \leq k \leq 2^n-1$ }
		\State {$\Delta B^n_k\gets B(t^{n+1}_{k+1})-B(t^n_k)$},
		\State{$\Delta B^{n+1}_{2k}\gets B(t^{n+1}_{2k+1})-B(t^{n+1}_{2k}), \Delta B^{n+1}_{2k+1}\gets B(t^{n+1}_{2k+2})-B(t^{n+1}_{2k+1})$},
		\State {$\Delta B^{n+1,a}_{2k} \gets \Delta B^{n+1}_{2k+1}$}, {$\Delta B^{n+1,a}_{2k+1} \gets \Delta B^{n+1}_{2k}$}
		\EndFor
		\State {$X_{n}(\cdot) \gets \textbf{Num\_Sol} (x,n, \{\Delta B_k^{n} \}_{1 \leq k \leq 2^{n}-1})$},
		\State {$X_{n+1}^f (\cdot) \gets \textbf{Num\_Sol} (x,n+1, \{\Delta B_k^{n+1} \}_{1 \leq k \leq 2^{n+1}-1})$},
		\State {$X_{n+1}^a (\cdot) \gets \textbf{Num\_Sol} (x,n+1, \{\Delta B_k^{n+1,a} \}_{1 \leq k \leq 2^{n+1}-1})$},
		\State {$\Delta_n\gets\frac{1}{2}\big( f(X_{n+1}^{f}(1))+f(X_{n+1}^{a}(1)) \big)%
			-f(X_n(1))$.}
		\EndProcedure
	\end{algorithmic}
\end{algorithm}
\begin{remark}
	We use the notation $X_{n+1}^f$ and $X_{n+1}^a$ consistently with \cite{giles2014} to represent the ``fine" and ``antithetic" solution on level $n+1$ versus the``coarse" solution $X_n$.
\end{remark}
\begin{lemma}
	Fixing $\mu\in\mathcal{L}_1$ and $\{{\mu}^{(n)}\}_{n\geq 1}\subset\mathcal{L}_1$, we have,as in \eqref{eq:deltal}, 
	\begin{equation}
	\mathbb{E}\Delta_n=
	\mathbb{E}f(X_{n+1}(1))-\mathbb{E}f(X_n(1)).
	\end{equation}
\end{lemma}
\begin{proof}
	Fixing any $n \geq 1$, since the Brownian increments are I.I.D., $X_{n+1}^{f}(1)$ and $X_{n+1}^{a}(1)$ produced by two recursions under the swapping of Brownian increments would follow the same marginal distribution, namely $\mathbb{E}f(X_{n+1}^{a}(1))=\mathbb{E}f(X_{n+1}^{f}(1))$.
\end{proof}
Now, we present a bound on $\mathbb{E}\Vert X_{n}(t)-X_{t}\Vert _{\infty }^{4}$, with  $\|\cdot \|_{\infty}$ norm. The reason for the fourth moment will become clear later.
\begin{lemma}\label{orderfirstlemma} 
	Fixing $\mu\in\mathcal{L}_1$, $\{{\mu}^{(n)}\}_{n\geq 1}\subset\mathcal{L}_1$. Let $X(\cdot)$ be the solution of the  SDE in \eqref{eq:SDE} and $X_n(\cdot)$ be the numerical approximation in \eqref{coarse}. Then,  for any  $\epsilon'>0$, we can find $0<\epsilon<\epsilon'$ ,  $0<\gamma<\frac{1}{3}$ and  $C>1$  such that, uniformly for all $0 \leq t \leq 1$,
	\begin{equation}  \label{manba}
	\mathbb{E}\|X_n(t)-X_t\|_{\infty}^4 \leq
	e^{CL_1}\cdot\Delta t_n^{2-\epsilon},
	\end{equation}
\end{lemma}

\begin{proof}
	The proof is in Section \ref{sec:proof}.
\end{proof}

\begin{coro}
	\label{expectationconverge} Under the assumptions of Lemma~\ref%
	{orderfirstlemma} and Assumption \ref{assump:2}, we have 
	\begin{equation}  \label{meanconvergence}
	\lim_{n \to \infty}\mathbb{E} f(X_n(1))=\mathbb{E}%
	f(X_1).
	\end{equation}
\end{coro}

\begin{proof}
	It follows from \ref{orderfirstlemma}, Assumption \ref{assump:2} and Cauchy-Schwarz inequality that 
	\begin{equation}\label{add1}
	\mathbb{E}|{f(X_n(1))-f(X_1)}|^2\leq L^2 \mathbb{E} \|X_n(1)-X_1\|_{\infty}^2\leq L^2  \sqrt{\mathbb{E}\|X_n(1)-X_1\|^4_{\infty}}.
	\end{equation}
	The quantity would converge to 0 as $n$ goes to infinity by Lemma~\ref{orderfirstlemma}.
\end{proof}

\begin{remark}
	For ${\mu}\in\mathcal{L}_1$ with a bounding number $L_1$, the results in \cite{giles2014} typically would show that $\mathbb{E}\Vert
	X_{n}(t)-X_{t}\Vert _{\infty }^{4}=O(\Delta t_{n}^{2})$, a standard error bound for numerical SDE based on Gronwall's inequality \cite{howard1998gronwall, kloeden2011numerical}. In particular, the bound has the form 
	\begin{equation}  \label{illstruuu}
	\mathbb{E}\|X_{n}(t)-X_t\|_{\infty} ^4\leq e^{C
		L_1^4}\Delta t_n ^{2},
	\end{equation}
	for some constant $C$. However, in our problem $\boldsymbol{\mu} $ is random and $%
	e^{\boldsymbol{L}_1^p}$ is not guaranteed to have a finite expectation for $p>1$. Thus, instead of using Gronwall's inequality, we explore the rough
	paths technique in \cite{blanchet2017} to develop an original bound as in
	\eqref{manba}, where we substitute the term $e^{C L_1^p}$ by $e^{CL_1}$ by
	giving up $\epsilon$ order from $\Delta t_n^2$ in \eqref{illstruuu}.
\end{remark}

\subsubsection{Construction and Properties of $Z(\mu)$}
\begin{defn}[{Construction of $Z(\mu)$}]
	Fixing $\mu\in\mathcal{L}_1$ and $\{{\mu}^{(n)}\}_{n\geq 1}\subset\mathcal{L}_1$, let $N \sim
	Geom(1-2^{-\theta}), N\geq 0$ be an  geometric random variable with $p_n\triangleq\mathbb{P}%
	(N=n)=(1-2^{-\theta})(2^{-\theta n})$ for some $\theta>0$. We defined $Z(\mu)$ to be:
	\begin{equation}  \label{eq:Z}
	Z(\mu) \triangleq f(X_{n_0}(1))+\frac{\Delta_{N+n_0}}{p_{N}},
	\end{equation}
	where $n_0 \geq 0$, $X_{n_0}(\cdot)$ is the base level estimator and $  \Delta_n$ is defined in \eqref{eq:Delta_given_b}.
\end{defn}
The choice of $\theta$ and $\gamma$ are specified in Section~\ref{sec:proof}. 
\makeatletter
\def\BState{\State\hskip-\ALG@thistlm}
\makeatother

\begin{algorithm}[H]\label{alg:z-generator}
	\caption{Generate  $Z(\mu)$ given $\theta$ and $\gamma$.}\label{euclid}
	\begin{algorithmic}[1]
		\Procedure{Unbiased\_Z}{$x, n_0$} with input $x \in \mathbb R^d$ and $n_0 \geq 0$. 
		\State $\textit{Generate }    N \gets Geom (1-2^{-\theta}), \text { and }   \boldsymbol V_i \gets \mathcal N(0, \Sigma_i)   \text{ for } 1 \leq i \leq \lfloor 2^{(N+n_0+1) \gamma} \rfloor. $ 
		\State ${\mu}^{(N+n_0+1)} \gets \sum_{i=1}^{\lfloor 2^{(N+n_0+1) \gamma} \rfloor}\frac{\lambda_i}{i^q} \boldsymbol{V}_i  \phi_i(\cdot) \text { and }{\mu}^{(N+n_0)} \gets \sum_{i=1}^{\lfloor 2^{(N+n_0) \gamma} \rfloor}\frac{\lambda_i}{i^q} \boldsymbol{V}_i  \phi_i(\cdot),$
		\State ${\mu}^{(n_0)} \gets \sum_{i=1}^{\lfloor 2^{(n_0) \gamma} \rfloor}\frac{\lambda_i}{i^q} \boldsymbol{V}_i  \phi_i(\cdot).$
		\For {$0 \leq k \leq 2^{n_0}-1$ }
		\State {$\Delta B^{n_0}_k\gets B(t^{n_0+1}_{k+1})-B(t^{n_0}_k)$},
		\EndFor
		\State {$X_{n_0}(\cdot) \gets Num\_Sol (x,n_0,\{\Delta B_k^{n_0} \}_{1\leq k \leq 2^{n_0}-1})$}
		\State {$\Delta_{N+n_0}\gets Delta\_Gen (x,N+n_0) \text{ and } p_N \gets (1-2^{-\theta})(2^{-\theta N})$},
		\State {$\textbf{Output } Z(\mu) \gets \frac{\Delta_{N+n_0}}{p_{N}}+f(X_{n_0}(t))$}
		\EndProcedure
	\end{algorithmic}
\end{algorithm}
\begin{remark}
	In practice, a larger value of $n_0$ gives lower variance of $Z$ at the cost of a 
	higher computational budget. Notice we also use the same Brownian path to generate $X_{n_0}(1)$ and $%
	\Delta_{N+n_0}$ in \eqref{eq:Z} to reduce the computational cost. 
\end{remark}
We now present several important properties of $Z$ starting with the unbiasedness.
\begin{lemma}
	\label{ieieieieieieie} Under the assumptions of Lemma~\ref%
	{orderfirstlemma} as well as Assumption \ref{assump:2} , we have
	\begin{equation*}
	\mathbb{E}[Z(\mu)]=u(x,1).
	\end{equation*}
\end{lemma}

\begin{proof}
	It follows from the definition of $Z$ that we know $\mathbb{E}[Z(\mu)]$ is equal to
	\begin{align*}
	\mathbb{E}f(X_{n_0}(1))+\mathbb{E}\frac{\Delta_{N+n_0}}{p_{N}}=&\mathbb{E}f(X_{n_0}(1))+\mathbb{E}_{N}[\mathbb{E}[\frac{\Delta_{N+n_0}}{p_{N}}|N]]\nonumber\\
	=&\mathbb{E}f(X_{n_0}(1))+\sum_{n=0}^{\infty}\frac{\mathbb{E}\Delta_{n+n_0}}{p_{n}}\cdot p_n\nonumber\\
	=& \lim_{n \to \infty}\mathbb{E} f(X_n(1))=\mathbb{E}f(X_1).
	\end{align*}
	The equality follows from the independence of $N$,  Equation \eqref{eq:deltal} and  Corollary~\ref{expectationconverge}.
\end{proof}
Next, instead a variance bound on $Z(\mu)$, we again provide a fourth moment bound of $Z(\mu)$ which becomes useful for proving the finite variance property of $W(\mu)$ later on.

\begin{lemma}
	\label{finitefourthmoment} 
	Fix $\mu\in\mathcal{L}_1$, $\{{\mu}^{(n)}\}_{n\geq 1}\subset\mathcal{L}_1$ and $L_1>1$, then for any $\delta'>0$, we can find $0<\delta<\delta'$ and $C>1$ such that
	\begin{equation}  \label{ffm1}
	\mathbb{E} \Delta_n^4 \leq e^{CL_1} \Delta
	t_n^{4-\delta},
	\end{equation} 
	\begin{equation}  \label{ffm2}
	\mathbb{E} |f(X_{n_0}(1))|^4\leq \mathcal{P}(L_1),
	\end{equation}
	for some polynomial function $\mathcal{P}(\cdot)$ such that $\mathcal{P}(x) > 1$ for $x>1$.
\end{lemma}

\begin{proof}
	The proof is in Section \ref{sec:proof}.
\end{proof}

\begin{lemma}
	\label{finitefourthmoment2} Under the assumptions of Lemma~\ref
	{finitefourthmoment}, if $\theta$ in the definition of $Z$ satisfies $3\theta<4-\delta $ for the $\delta$ in Lemma~\ref{finitefourthmoment},
	then the estimator $Z$ defined in \eqref{eq:Z}
	satisfies 
	\begin{equation}
	\mathbb{E}[Z^4(\mu)] \leq e^{CL_1},
	\end{equation}
	for some constant $C>1$.
\end{lemma}

\begin{proof}
	The elementary inequality 
	\begin{equation}\label{elementary inequality}
	|{\sum_{n=1}^N a_n}|^p \leq N^{p-1}\sum_{n=1}^N |{a_n}|^p,
	\end{equation}
	shows that $\mathbb{E}Z^4(\mu) $ is bounded by $8\mathbb{E}\frac{\Delta_{N+n_0}^4}{p_{N}^4}+8\mathbb{E} |f(X_{n_0}(1))|^4$ and
	\begin{align*}
	8\sum_{n=0}^{\infty}\frac{\mathbb{E}\Delta_{N+n_0}^4}{p_{n}^3}+8\mathbb{E} |f(X_{n_0}(1))|^4 \leq 8\big(\frac{e^{CL_1}}{(1-2^{-\theta})^3} \sum_{n=0}^{\infty}\frac{\Delta t_n^{4-\delta}}{\Delta t_n^{3\theta}}+\mathcal{P}(L_1)\big)
	\leq  e^{C'L_1}. 
	\end{align*}
	for some constant $C'>1$ chosen appropriately. In particular, the last inequality follows from $4-\delta >3\theta$, $L_1>1$ and the fact that  $\mathcal{P}(|x|)<e^{cx}$ for some appropriately chosen $c$ if $x>1$. The second inequality follows from Lemma~\ref{finitefourthmoment}. 
\end{proof}
Finally, we show that $Z(\mu)$ has finite expected computational cost. Formally, if we use ${%
	cost^{Z}}$ to denote the computational cost for generating $Z$ and $cost_n$ for $X_n(1)$, then 
\begin{equation}  \label{eq:total-cost}
{cost_{Z}} = cost_{n_0}+cost_{N+n_0} + 2 cost_{N+n_0+1},
\end{equation}
for the computation of $X_{n_0}(1),X_{N+n_0}(1),X^f_{N+n_0+1}$ and $X^a_{N+n_0+1}$ in $Z$. 
\begin{lemma}
	\label{computecost} 
	Let $\theta $ and $\gamma $ be chosen so that $\theta>1+\gamma$, then the
	computational cost for generating $Z$ has a finite expectation. That is, 
	\begin{equation}
	\mathbb{E}({cost}_Z )<\infty .
	\end{equation}
\end{lemma}
\begin{proof}
	Consider the  ${cost_n}$ for generating $X_n(1)$. For fixed $n$, we need to generate $2^n$ Brownian increments and $2^{\gamma n}$ of $\boldsymbol{V}_i$ for $\boldsymbol{\mu}^{(n)}$. Then, to compute $X_n(1)$, we need to $2^n$  recursions to obtain $X_n(1)$ from start $x$ and each iteration requires $O(2^{\gamma n})$ computation on $\phi_1(X_n(t^n_k)),...,\phi_{2^{\lfloor \gamma n \rfloor}} (X_n(t^n_k))$ in evaluating $\boldsymbol{\mu}^{(n)}(X_n(t^n_k))$. Thus $cost_n$ satisfies 
	\begin{equation}
	{cost_n}=O(2^{(1+\gamma)n}) \leq C2^{(1+\gamma)n},
	\end{equation} 
	for some constant $C$.
	Therefore, from \eqref{eq:total-cost} and $p_n \leq 2^{-\theta n}$, we bound $\mathbb{E}({cost_Z})$ by
	\begin{align}\label{cost111}
	&\mathbb{E}(cost_{n_0})+\mathbb{E}(cost_{N+n_0})+2\mathbb{E}(cost_{N+n_0+1}) \nonumber\\
	\leq& C 2^{(1+\gamma) n_0}(1+\sum_{n=0}^\infty 2^{(1+\gamma-\theta)n}+2^{1+\gamma}\sum_{n=0}^\infty 2^{(1+\gamma-\theta)n})<\infty
	\end{align}
	due to the assumption $\theta>1+\gamma$.
\end{proof}

\subsection{Step 2: Unbiased Estimator $W(\mu)$ for $G(\mathbb{E}[Z_{1}(\mu)],...,%
	\mathbb{E}[Z_k(\mu)])$}

\subsubsection{Construction of $W(\mu)$}

After the construction of $Z(\mu)$, we construct $W(\mu)$ that 
\begin{equation}
\mathbb{E}%
W(\mu)=G(\mathbb{E} Z_1(\mu), ... ,\mathbb{E}Z_k(\mu)) ,
\end{equation}
with method recently developed in \cite{blanchet2015unbiased}. For the ease of presentation, we only
construct unbiased estimators $W(\mu)$ for $G(\mathbb{E}%
(Z(\mu))$ for one dimensional $G(\cdot):\mathbb{R}\rightarrow \mathbb{R}$. The case for $G(\cdot):\mathbb{R}^{k}\rightarrow \mathbb{R}$ can be generalized and we leave the details in Algorithms.

\begin{defn}[{Construction of $W(\mu)$}]
	Fixing $\mu\in\mathcal{L}_1$ and $\{{\mu}^{(n)}\}_{n\geq 1}\subset\mathcal{L}_1$, let $\{Z_{j}(\mu)\}_{j \geq 1} $ be I.I.D. copies of random variables $Z(\mu)$. Define $\widetilde{\Delta}_n$ to be
	\begin{equation}  \label{finaltri}
	\widetilde{\Delta}_n\triangleq G(\frac{\sum_{j=1}^{2^{n+1}} Z_{j}(\mu)}{2^{n+1}})-%
	\frac{1}{2}\Big(G(\frac{\sum_{j=1}^{2^{n}} Z_{j}(\mu)}{2^{n}})+G(\frac{%
		\sum_{j=2^n +1}^{2^{n+1}} Z_{j}(\mu)}{2^{n}})\Big).
	\end{equation}
	Then, fix $n_1 \geq 1$ and define  $W(\mu)$ to be
	\begin{equation}  \label{finalw}
	W(\mu)=\frac{\widetilde{\Delta}_{\widetilde {N}+n_1}}{\widetilde{p}_{\widetilde{N}%
	}}+G(\frac{\sum_{j=1}^{2^{n_1}} Z_{j}(\mu)}{2^{n_1}}),
	\end{equation}
	where $\widetilde{N} \sim Geom(1-2^{-1.5})$ with $%
	\widetilde{p}_{n}\triangleq\mathbb{P}(\widetilde{N}=n)=2^{-1.5n}(1-2^{-1.5})$
	.
\end{defn}

\begin{algorithm}[H]
	\caption{Generate $\rho(\cdot)$ given $\{ Z_{ij}\}$.}\label{euclid1.5}
	\begin{algorithmic}[1]
		\Procedure{$\rho$}{$a,b$} with integers $a<b$. 
		\State $\rho(a,b) \gets G\bigg(\frac{1}{b-a+1}{\sum\limits_{j=a}^{b}Z_{1j}},...,\frac{1}{b-a+1}{\sum\limits_{j=a}^{b}Z_{kj}}\bigg)$
		\EndProcedure
	\end{algorithmic}
\end{algorithm}
\begin{algorithm}[H]\label{alg:w-generator}
	\caption{Generate  $W(\mu)$ given $\theta$ and $\gamma$.}\label{euclid2}
	\begin{algorithmic}[1]
		\Procedure{Unbiased\_W}{$x_1,...,x_k, n_0,n_1$} with input $x_i \in \mathbb R^d$ for $1 \leq i \leq k$, $n_0 \geq 0$ and $n_1>0$. 
		\State $\textit{Generate }    \widetilde N \gets Geom (1-2^{-1.5})$
		\For {$1 \leq i \leq k$}
		\For {$1 \leq j \leq 2^{N+n_1+1}$}
		\State $\textit{ Generate }  Z_{ij} \gets \textsc{Unbiased\_Z}(x_i,n_0) $ 
		\EndFor
		\EndFor 
		\State $p_{\widetilde N} \gets 2^{-1.5\widetilde N}(1-2^{-1.5})$
		\State $\widetilde{\Delta}_{\widetilde{N}+n_1}\gets\rho(1,2^{\widetilde{N}+n_1+1})-\frac{1}{2}\big(\rho(1,2^{\widetilde{N}+n_1})+\rho(2^{\widetilde{N}+n_1}+1,2^{\widetilde{N}+n_1+1})\big)$
		\State {$\textbf{Output } W(\mu) \gets \frac{\widetilde{\Delta}_{\widetilde {N}+n_1}}{p_{\widetilde{N}}}+\rho(1,2^{n_1})$}
		\EndProcedure
	\end{algorithmic}
\end{algorithm}
\begin{remark}\label{finitevsample}
	Notice that if we denote $N_{ij}$ to be the geometric random variable generated during the construction of $Z_{ij}$ and let 
	\begin{equation*}
	m=\max\{N_{ij}, 1 \leq i \leq k, 1\leq j\leq 2^{\widetilde{N}+n_1+1}\} \quad \text{and}\quad M=\lfloor2^{(m+n_0+1)\gamma }\rfloor,
	\end{equation*}
	then we only need $V_1,...,V_{M}$ because they are sufficient for Algorithm 1.
\end{remark}

\subsubsection{Properties of $W(\mu)$}
We show properties of $W(\mu)$ starting with the unbiasedness.
\begin{lemma}
	\label{ieieieieieie} Under the assumptions of Lemma~\ref%
	{orderfirstlemma} as well as Assumptions \ref{assump:2}-\ref{assump:3},
	\begin{equation}\label{firstclaim}
	\mathbb{E}W(\mu)=G(\mathbb{E} Z(\mu)).
	\end{equation}
\end{lemma}

\begin{proof}
	According to Lemma \ref{finitefourthmoment2} and the strong law of large numbers(SLLN),
	\begin{equation}
	\lim\limits_{n \to \infty}\mathbb{E}|{\frac{\sum_{j=1}^{2^n} Z_{j}(\mu)}{2^n} -\mathbb{E} Z(\mu)}|=0,
	\end{equation}
	which implies, by Assumption~\ref{assump:3} on the bound of $\|\frac{\partial G}{\partial x_{i}}\|_{\infty}$,
	\begin{equation}
	\lim\limits_{n \to \infty} \mathbb{E}G(\frac{\sum_{j=1}^{2^n} Z_{j}(\mu)}{2^n}) = G(\mathbb{E}Z(\mu))
	\end{equation} as $n \to \infty$. Now, sincev$\mathbb{E}\widetilde{\Delta}_n=\mathbb{E}G(\frac{\sum_{j=1}^{2^{n+1}} Z_{j}(\mu)}{2^{n+1}})-\mathbb{E}G(\frac{\sum_{j=1}^{2^{n}} Z_{j}(\mu)}{2^{n}})$,
	the rest of the proof follows as in the proof of Lemma~\ref{ieieieieieieie}. 
\end{proof}
Second, we proceed to show that $W(\mu)$ also has finite variance.

\begin{lemma}
	\label{deltadelta2} Under the assumptions of Lemma~\ref%
	{orderfirstlemma} and Assumptions \ref{assump:2}-\ref{assump:3},  $\widetilde{\Delta}_n$
	satisfies 
	\begin{equation}
	\mathbb{E}(\widetilde{\Delta}_n)^2 \leq e^{{C}%
		L_1} \Delta t_n^{2}
	\end{equation}
	where ${C}>1$ is some fixed constant.
\end{lemma}

\begin{proof}
	Define $S(a,b) \triangleq \frac{\sum_{j=a}^b Z_j}{b-a+1}$ and $S^k(a,b)=(S(a,b)-\mathbb Z(\mu))^k$. Then, a second order Taylor expansion of $G(\cdot)$ on $\mathbb{E}Z(\mu)$ gives
	\begin{align} \widetilde{\Delta}_n=&G(S(1,2^{n+1}))-\frac{1}{2}\Big(G(S(1,2^n))+G(S(2^n+1,2^{n+1}))\Big) \nonumber \\
	=&G^{'}(\mathbb{E}Z(\mu))(S(1,2^{n+1})-\frac{1}{2}\Big(S(1,2^n)+S(2^n+1,2^{n+1})\Big)\nonumber\\
	&+\frac{G^{''}(\xi_1)}{2}S^2(1,2^{n+1})-\frac{G^{''}(\xi_2)}{4} S^2(1,2^n) -\frac{G^{''}(\xi_3)}{4} S^2(2^{n}+1,2^{n+1}),
	\end{align}
	where term $G(\mathbb E Z(\mu))$ cancels out and as in the mean value theorem, $\xi_1$ stands for a random variable between $\mathbb E Z(\mu)$ and $S(1,2^{n+1})$, similarly $\xi_2$ for $S(1,2^{n})$ and $\xi_3$ for $S(2^n+1,2^{n+1})$.
	Thus, it follows from \eqref{elementary inequality} and Assumption \ref{assump:3}, we have
	\begin{equation}\label{abovedelta}
	|{\widetilde{\Delta}_n}|^2 \leq \frac{3L^2}{4}\big( S^4(1,2^{n+1})+\frac{1}{4} S^4(1,2^n)+\frac{1}{4}S^4(2^n+1,2^{n+1})\big).
	\end{equation}
	However, the $(Z_j(\mu)-\mathbb{E}Z(\mu))$  are I.I.D. with mean 0. In particular, when we write out the expansion in \eqref{abovedelta} and take expectation, the terms with odd power will vanish 
	\begin{align}
	\mathbb{E}[(Z_i(\mu)-\mathbb{E}Z(\mu))^2(Z_j(\mu)-\mathbb{E}Z(\mu))(Z_k(\mu)-\mathbb{E}Z(\mu))]=&0 \nonumber\\ \mathbb{E}[(Z_i(\mu)-\mathbb{E}Z(\mu))^3(Z_j(\mu)-\mathbb{E}Z(\mu))]=&0 \nonumber\\
	\mathbb{E}[(Z_i(\mu)-\mathbb{E}Z(\mu))(Z_j(\mu)-\mathbb{E}Z(\mu))(Z_k(\mu)-\mathbb{E}Z(\mu))(Z_l(\mu)-\mathbb{E}Z(\mu))]=&0.
	\end{align}
	Thus, taking expectation in \eqref{abovedelta} gives $	\mathbb{E}{\widetilde{\Delta}_n}^2 $ is bounded by
	\begin{align}
	&\frac{3L^2}{2^{4n+4}}\mathbb{E} 
	\bigg[\frac{1}{4}\big(\sum_{j=1}^{2^{n+1}}Z_j(\mu)-\mathbb{E}Z(\mu)\big)^4+\big(\sum_{j=1}^{2^{n}}Z_j(\mu)-\mathbb{E}Z(\mu)\big)^4+\big(\sum_{j=2^n+1}^{2^{n+1}}Z_j(\mu)-\mathbb{E}Z(\mu)\big)^4\bigg] \nonumber\\
	\leq& C\binom{2^{n+1}}{2}{{2^{-4n}}}\cdot \mathbb{E}\big(Z(\mu)-\mathbb{E}Z(\mu)\big)^4
	\end{align}
	for some constants $C>1$ since $\mathbb E(Z_j(\mu)-\mathbb E Z(\mu))^2(Z_i(\mu)-\mathbb E Z(\mu))^2 \leq \mathbb E(Z_j(\mu)-\mathbb E Z(\mu))^4$. However, it can be shown that $\binom{n}{2}= O(n^2)$.
	Thus, we have
	\begin{equation}
	\binom{2^{n+1}}{2}{{2^{-4n}}}\leq C\Delta t_n^2,
	\end{equation}
	for some constant $C>1$. Furthermore,  we can bound $\mathbb{E}\big(Z(\mu)-\mathbb{E}Z(\mu)\big)^4$ using $\mathbb{E}Z^4(\mu)$ and its bound from Lemma~\ref{finitefourthmoment2} to obtain \begin{equation}
	\mathbb{E}\big(Z(\mu)-\mathbb{E}Z(\mu)\big)^4\leq e^{CL_1}
	\end{equation}
	for some constant $C>1$. Finally, we conclude there is some constant $C>1$ such that 
	\begin{equation}
	\mathbb{E}(\widetilde{\Delta}_n)^2 \leq e^{{C}%
		L_1} \Delta t_n^{2}.
	\end{equation}
\end{proof}

\begin{lemma}
	\label{varianceW} Under the assumptions of Lemma~\ref%
	{orderfirstlemma} and Assumptions \ref{assump:2}-\ref{assump:3},  $W$ satisfies 
	\begin{equation}
	\mathbb{E}W^2(\mu)\leq e^{CL_1}
	\end{equation}
	for some constant $C>1$.
\end{lemma}

\begin{proof}
	Using bounds on the fourth moment of $Z$ in Lemma \ref{finitefourthmoment2}, the linear growth condition of $G(\cdot)$ in Assumption~\ref{assump:3} and the Cauchy-Schwarz inequality, we can bound
	\begin{align}\label{extracost2}
	\mathbb{E}|G(\frac{\sum_{j=1}^{2^{n_1}} Z_{j}(\mu)}{2^{n_1}})\big|^2 \leq& \mathbb{E}(|G(0)|+L|\frac{\sum_{j=1}^{2^{n_1}} Z_{j}(\mu)}{2^{n_1}}|)^2  \nonumber\\
	\leq & |G(0)|^2+2|G(0)|L\mathbb{E}|Z(\mu)|+L^2\mathbb{E}Z^2(\mu) \leq C+Ce^{CL_1}
	\end{align}
	for some constant $C>1$. Now, using \eqref{elementary inequality}, \eqref{extracost2} and \ref{deltadelta2}, we have
	\begin{align}\label{justabove}
	\mathbb{E}W^2(\mu)\leq &  2\mathbb{E}\Big( \frac{\widetilde{\Delta}_{N+n_1}^2}{\widetilde{p}^2_{N}}+|G(\frac{\sum_{j=1}^{2^{n_1}} Z_{j}}{2^{n_1}})|^2 \Big)  \nonumber\\
	=& 2\sum_{n=0}^{\infty} \frac{\mathbb{E}\widetilde{\Delta}_{n_1+n}^2}{\widetilde{p}_{n}}+2\mathbb{E}|G(\frac{\sum_{j=1}^{2^{n_1}} Z_{j}}{2^{n_1}})|^2  \nonumber \\
	\leq& \frac{2 e^{{C}L_1}}{(1-2^{-1.5})} \sum_{n=0}^{\infty} \frac{2^{-2n}}{ 2^{-1.5n}}+2C+2Ce^{CL_1} \leq e^{C^{\prime}L_1} 
	\end{align}
	for some appropriately chosen $C^{\prime}>1$. The last line follows from the fact that for any $a,b$ and $c$, we can find $d$ such that $a+ce^{bx}<e^{dx}$ for $x>1$.
	
\end{proof}
Finally, we discuss the computational cost for generating $W(\mu)$. Denote the cost by $cost_{W}$. Since
we use the first $2^{{n_{1}}}$ samples of $Z_{j}$ in the construction of 
$\widetilde{\Delta }_{\widetilde{N}+n_{1}}$ to construct $%
G(\sum_{j=1}^{2^{n_{1}}}Z_{j}/2^{n_{1}})$, we only consider the cost induced
by term $\widetilde{\Delta }_{\widetilde{N}+n_{1}}$, namely
\begin{equation}
cost_{W}=\sum_{j=1}^{2^{\widetilde{N}+n_{1}+1}}cost_{Z_{j}}.
\label{compcostw}
\end{equation}

\begin{lemma}
	\label{computecost2} The total expected computational cost satisfies
	\begin{equation}
	\mathbb{E}({cost_W})< \infty.
	\end{equation}
\end{lemma}

\begin{proof}
	Using Wald's identity, we have
	\begin{equation}
	\mathbb{E}({cost_W}) =\mathbb{E}(2^{\tilde{N}+n_1+1})\mathbb{E}({cost}%
	_Z)=2^{n_1+1}\bigg(\sum_{n=0}^{\infty}2^{-0.5 n}(1-2^{-1.5})\bigg)\mathbb{E}({cost}_Z)<\infty,
	\end{equation}
	where $\mathbb{E}({cost}_Z)<\infty$ follows from Lemma~\ref{computecost}.
\end{proof}

\subsection{Proof of Theorem \protect\ref{thm:main}}
\begin{proof}	\label{sec:proof-main}
	We simulate $W(\mu)$ according to Algorithm \ref{euclid2}. Since ${\mu}(\cdot)\in\mathcal{L}_1$ almost surely by Lemma ~\ref{assump1}, it follows from Lemma~\ref{ieieieieieie} that $W(\mu)$ satisfies 
	\begin{equation*}
	\mathbb{E}[W(\mu)]=G(\mathbb{E}[Z_{1}(\mu)],...,%
	\mathbb{E}[Z_k(\mu)])=G(u(x_{1},t_{1}),...,u(x_{k},t_{k})),
	\end{equation*}%
	by Lemma ~\ref{ieieieieieieie} and Theorem~\ref{fkt},
	where $\mu(\cdot)\in\mathcal{L}_1$ is fixed. Thus, when $\boldsymbol{\mu}$ is random,
	\begin{align*}
	\mathbb{E}[W]=\mathbb{E}_{\mu\sim\boldsymbol{\mu}}[\mathbb{E}[W(\mu)]]
	=\mathbb{E}[G(\boldsymbol{u}(x_{1},t_{1}),...,\boldsymbol{u}(x_{k},t_{k}))]=\nu,
	\end{align*}
	which proves the unbiasedness of $W$. To show the finite variance property of $W$, we use Lemma~\ref{lemma L1} and Lemma~\ref{varianceW} to obtain,
	\begin{align*}
	\mathbb{E}W^{2}=\mathbb{E}_{\mu\sim\boldsymbol{\mu}}[\mathbb{E}[W^2(\mu)]]\leq \mathbb{E}_{\mu\sim\boldsymbol{\mu}}%
	[ e^{C\boldsymbol{L}_1}] <\infty.
	\end{align*}%
	Finally, the finite expected computational
	cost follows directly from Lemma~\ref{computecost2}.

\end{proof}

\begin{remark}
	To digress, if $\sigma$ is not bounded, one can localize $\sigma$ by constructing $\sigma^{N}$: 
	\begin{equation}
	\begin{cases}
	\sigma^N(x)=\sigma(x) & when\ \ \ \|x\| \leq N \\ 
	\sigma^N(x)=0 & when\ \ \ \|x\| > N+1%
	\end{cases}.%
	\end{equation}
	Denote estimator by $W^N$ when generated under $\sigma^N(\cdot)$. Then, by adding randomization with $N^{\prime }$  again being geometric random variable,
	\begin{equation}
	\widetilde{W} \triangleq W^{n_2}+\frac{W^{N^{^{\prime
			}}+1+n_2}-W^{N^{^{\prime }}+n_2}}{p_{N^{^{\prime }}}},
	\end{equation}
\end{remark}

\section{Simulation}\label{sec:sim} 
\paragraph{Example 1}

We first introduce an example to check the unbiasedness of our estimator. Consider the one-dimensional SDE known as the Ornstein-Uhlenbeck Process  \cite{kloeden2011numerical}:
\begin{equation}  \label{ouprocess}
\begin{cases}
dX_t=& -\boldsymbol{\alpha} X_t dt+ dB_t\qquad \text{for $t
	\geq 0$} \\ 
X_0  =&0%
\end{cases},%
\end{equation}
where $\boldsymbol{\alpha} \in \mathbb{R}$ is a random. Given realizations of $\alpha(\omega)$, the solution can be found exactly,
\begin{equation}
X_1=e^{-\alpha t} \int_0^1 e^{\alpha s} \ dB_s  .
\end{equation}
Consequently, given realizations of $\alpha(\omega)$, using It\^{o}'s isometry, it can be shown that $X_1$ is
Gaussian with mean 0 and variance $(2\alpha(\omega))^{-1}({1-e^{-2\alpha(\omega)}})$. For simulation, we set $\alpha(\omega)$ to be Gaussian with mean 1
and variance $0.05^2$ along with $f(x)=x^2$, $G(x)=e^{-x^2}$. Then, it follows from calculation that, 
\begin{align*}  
\mathbb{E}[f(X_1)|\alpha]=& \frac{1-e^{-2\alpha}}{
	2\alpha},\nonumber\\
\mathbb{E}[G(\mathbb{E}[f(X_1)|\alpha])]=&\int_{-\infty}^{\infty} 
\frac{1}{\sqrt{2\pi\cdot 0.05^2}}\cdot \ e^{-\frac{(x-1)^2}{2\cdot 0.05^2}}
\cdot \ e^{-(\frac{({1-e^{-2x}})}{2x})^2} dx\approx 0.8291.
\end{align*}
To check the unbiasedness property of $Z$, we first fix $\alpha=1$ in simulation so that $\mathbb{E}[f(X_1)|\alpha=1]\approx 0.4323$ .
Picking $%
n_0=5$ as the base level, we
generate 10,000 copies of $Z$ with $\alpha=1$. A
sample mean of $0.4303$ is obtained to compare with its true mean $0.4323$, as in Figure
1. Then, we pick $n_1=5$ and generate 10,000 copies of $W$ to obtain a sample mean of $0.8323$
to while the true mean is $0.8291$, as in Figure 2a. Furthermore, in Figure 2b, we generate 10000 copies of unbiased estimators of $G(u(x,1))$ using the multilevel Monte Carlo estimator based on numerical PDE as proposed in \cite{li2016multilevel}. In both cases, the
sample size is 10,000 and the difference between sample mean and true mean is
well within the margin dictated by CLT, $\frac{1}{\sqrt{ N_{copy}}}=\frac{1}{\sqrt{10000}}=0.01$. Overall, the findings are consistent with
our theoretical results on the unbiasedness.

\begin{figure}[tbp]
	\begin{center}
		\includegraphics[width=0.45\linewidth]{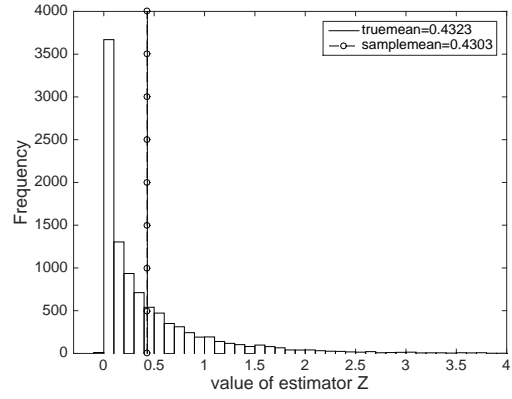}
		\caption{Histogram of Estimator $Z$ when
			$\alpha=1$}\label{}
	\end{center}
\end{figure}

\begin{figure}[tbp]
	\centering  
	\subfigure[Histogram of Estimator
	$W$]{\includegraphics[width=0.45\linewidth]{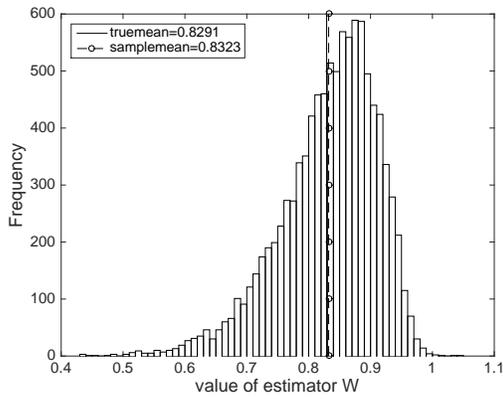}} 
	\subfigure[ Estimators based on Numerical PDE]{\includegraphics[width=0.45\linewidth]{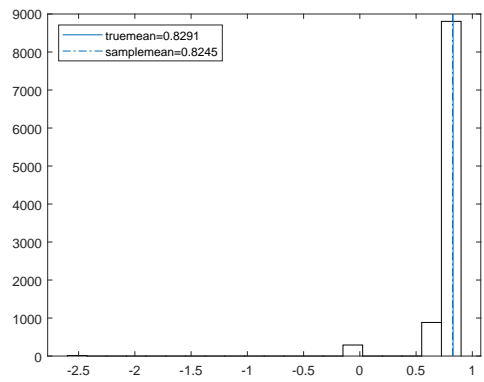}}
	\caption{\text{Comparsion of Multilevel Estimators based on Antithetic Numerical SDE or Numerical PDE}}
\end{figure}

{}\ {}\ 

\paragraph{Example 2}
In this example, we consider the more complicated SDE:
\begin{equation}
\begin{cases}
dX_{t}& =-\boldsymbol{\mu} (X_{t} )dt+\cos(X_{t}) dB_{t}\qquad \text{for $t\geq 0$} \\ 
X_{0} &=0,%
\end{cases}
\label{example2}
\end{equation}%
where $\boldsymbol{\mu}(x)=\sum\limits_{i=1}^{\infty }i^{-4} \sin(ix)
\boldsymbol{V}_{i}$ and we
compare the proposed method with the standard Monte Carlo method with
bias. We take $\gamma =\frac{1}{3}$ and $\theta =\frac{4}{3}$ for simplicity ( the detailed discussion in Section~\ref{sec:proof}). Similar to the previous example, we
take $n_{0}=n_{1}=5$. We generate $10,000$ copies of our estimator
and compare it with $10,000$ copies of a standard Monte Carlo estimator where
we remove the debiasing part $\frac{\Delta _{N}}{p_{N}}$ in 
\emph{both} estimator $Z$ and $W$.  As a result,
using the CLT, we compute a 95\%
confidence interval
$[0.4610,0.4656]$ for our estimator while we obtain an interval $[0.5189,0.5255]$ for the standard
Monte Carlo estimator. As we can see, these two intervals are not overlapping, suggesting that the
standard Monte Carlo estimator has shown a significant bias.

\appendix

\section{Proofs}\label{sec:proof} 
In this section, we present the proofs for Lemma~\ref%
{orderfirstlemma},  Lemma~\ref{finitefourthmoment} and Lemma~\ref{lemma L1}. The proof for all the
supporting lemmas are provided in the Appendix.
\subsection{Definitions and supporting lemmas}
 To prove Lemma~\ref%
{orderfirstlemma}, we introduce several definitions and supporting lemmas.

\begin{defn}
	\label{alni} Let $\epsilon $ to be a positive constant small enough to
	satisfy 
	\begin{equation}
	\epsilon <\frac{1}{144}\ \ \ \ \text{and}\ \ \ \ \epsilon <\frac{1}{36}(%
	\frac{1}{6}-12\epsilon )(q-4),  \label{epsiloncondition}
	\end{equation}%
	where $q>4$ is from Assumption~\ref{assump:1}, so that we can define positive quantities 
	\begin{equation}
	\alpha \triangleq \frac{1}{2}-\epsilon ,\ \ \beta \triangleq \frac{1}{2}%
	+2\epsilon ,\ \ \gamma \triangleq \frac{1}{3}-12\epsilon ,\ \ \ \theta
	\triangleq \frac{4}{3}-\frac{23}{2}\epsilon \ \ \ and\ \ \ \delta \triangleq
	33\epsilon  \label{parameter}
	\end{equation}%
	It is easy to check that the following important inequalities are satisfied :
	\begin{align}
\gamma \geq& \frac{1}{4}, \quad (3+\frac{q-4}{2}%
)\gamma >1, \quad 8(2\alpha
-\beta )>4-\delta>0 , \nonumber\\
4-\delta >&3\theta>0 \quad(\text{as in Lemma~\ref{finitefourthmoment2}})\quad \text{and} \quad \theta >1+\gamma>0 \quad (\text{as in Lemma~\ref{computecost}}).
\end{align}
\end{defn}
\begin{defn}\label{add15}
	\label{jose0} For a standard one-dimensional Brownian motion $B(t)$ on $[0.1]$, let $\alpha $ and $%
	\beta $ be defined as in Definition \ref{alni}. Then, define 
	\begin{equation}
	\Vert B\Vert _{\alpha }\triangleq \sup_{0\leq s<t\leq 1}\frac{\Vert
		B(t)-B(s)\Vert _{\infty }}{\lvert t-s\rvert ^{\alpha }}\qquad \text{and}%
	\qquad \Vert A\Vert _{2\alpha }\triangleq \sup_{0\leq s<t\leq 1}\max_{1\leq
		i,j\leq d^{\prime}}\frac{\lvert A_{i,j}(s,t)\rvert }{\lvert t-s\rvert ^{2\alpha }}
	\end{equation}%
	\begin{equation}
	\Vert \tilde{A}\Vert _{2\alpha }\triangleq \sup_{0\leq s\leq t\leq
		1}\max_{1\leq i,j\leq d^{\prime}}\frac{\lvert \widetilde{A}_{i,j}(s,t)\rvert }{\lvert
		t-s\rvert ^{2\alpha }}\qquad \text{and}\qquad \Gamma _{\widetilde{R}%
	}\triangleq \sup_{n}\sup_{\substack{ 0\leq s\leq t\leq 1  \\ s,t\in D_{n}}}%
	\max_{1\leq i,j\leq d^{\prime}}\frac{\lvert \widetilde{R}_{i,j}^{n}(s,t)\rvert }{%
		\lvert t-s\rvert ^{\beta }\Delta t_{n}^{2\alpha -\beta }},
	\end{equation}%
	where $D_{n}$ is the dyadic rationals that are multiples of $\frac{1}{%
		2^{n}}$ in $[0,1]$, and for $1 \leq i,j \leq d^{\prime}, i\neq j$,
	\begin{equation}
	A_{i,j}(s,t)\triangleq \int_{s}^{t}(B_{i}(u)-B_{i}(s))dB_{j}(u)~~\text{and}~~%
	\widetilde{A}_{i,i}(s,t)\triangleq {A}_{i,i}(s,t)=\frac{%
		(B_{i}(t)-B_{i}(s))^{2}-(t-s)}{2},
	\end{equation}%
	\begin{equation}
	\widetilde{A}_{i,j}(s,t)\triangleq \frac{%
		(B_{i}(t)-B_{i}(s))(B_{j}(t)-B_{j}(s))}{2}~~\text{and}~~\widetilde{R}%
	_{i,j}^{n}(t_{l}^{n},t_{m}^{n})\triangleq
	\sum_{k=l+1}^{m}\{A_{i,j}(t_{k-1}^{n},t_{k}^{n})-\widetilde{A}%
	_{i,j}(t_{k-1}^{n},t_{k}^{n})\},
	\end{equation}
	some of which we have already defined in Definition~\ref{cdefn}.
\end{defn}
\begin{defn}[\textbf{Notation}]\label{PandC}
Throughout the proof section, we will use $C$ to represent any constant greater than 1 (i.e., $C>1$) and use $\mathcal{P}(\cdot)$ to represent any polynomial function from $\mathbb{R}^n\rightarrow\mathbb{R}$ where $n\geq 1$ such that $\mathcal{P}({x})>1$ for any ${x}_1 >1$, $x_i \geq 0$ for $2 \leq i \leq n$, and $x=(x_1,...,x_n)$. We will  simply write this as $\mathcal{P}(x)>1$ for $x>1$ and it will not affect our analysis.

 It is straightforward to verify that for any $\mathcal{P}_1(\cdot),\mathcal{P}_2(\cdot)$ and $n\geq 0$, we can find some $\mathcal{P}_3(x)$ that
\begin{align}
(\mathcal{P}_1(x) )^n<& \mathcal{P}_3(x) \nonumber\\
\mathcal{P}_1(x)+\mathcal{P}_2(x)<&\mathcal{P}_3(x)\nonumber\\
\mathcal{P}_1(x)\cdot\mathcal{P}_2(x)<&\mathcal{P}_3(x)\nonumber\\
\mathcal{P}_2(\mathcal{P}_1(x))<&\mathcal{P}_3(x).
\end{align} 
\end{defn}

\begin{lemma}[\textbf{Supporting Lemma}]
	\label{jose1} The quantities $\|B\|_{\alpha},\|A\|_{2\alpha}$, $\| 
	\tilde{A}\|_{2\alpha}$ and $\Gamma_{\widetilde{R}} $ defined in Definition~\ref{add15}
	have moments of arbitrary order.
\end{lemma}
\begin{lemma}[\textbf{Supporting Lemma}]
	\label{note lemma3} Let $X_{n}(\cdot )$ be the 
	discretization scheme in Definition~\ref{coarse}  generated under $\mu ^{(n)}(\cdot ) \in \mathcal{L}_1$ with the bounding number $L_1>1$ and Brownian motion $B(t), 0\leq t \leq 1$. Then, we can find some fixed
	polynomial function $\mathcal{P}(x )>1$ for $x>1$ $:\mathbb{R}^3 \rightarrow \mathbb{R}$ such that 
	\begin{equation*}
	\Vert X_{n}(t)-X_{n}(r)\Vert _{\infty }\leq \mathcal{P}(L_{1},\Vert B\Vert _{\alpha },\Vert \widetilde{A}\Vert _{2\alpha })\lvert
	t-r\rvert ^{\alpha }
	\end{equation*}%
	for $0\leq r \leq t \leq1$ and for all $n\geq 0$.
\end{lemma}

\begin{lemma}[\textbf{Supporting Lemma}]
	\label{note lemma4} Let $X^{\mu}_{n}(\cdot )$ be the 
	discretization scheme modified from Definition~\ref{coarse}  generated under $\mu (\cdot ) \in \mathcal{L}_1$ (instead of $\mu^{n}(\cdot)$) 
	with the bounding number $L_1>1$ and Brownian motion $B(t), 0\leq t \leq 1$. Also, let $X_t, 0\leq t \leq 1$ be the solution of SDE in %
	\eqref{eq:SDE}. Then, we can find some fixed polynomial $\mathcal{P}(x )>1$ for $x>1$$:\mathbb{R}^4 \rightarrow \mathbb{R}$
	such that 
	\begin{equation}
	\|X_n^{\mu}(t)-X_t \|_{\infty}\leq \mathcal{P}(L_1,\|B\|_{\alpha},
	\|A\|_{2\alpha},\Gamma_{\widetilde{R}})\Delta t_n^{2\alpha-\beta}
	\end{equation}
	for all $n\geq0$ and $0\leq t \leq 1$.
\end{lemma}

\begin{lemma}[\textbf{Supporting Lemma}]
	\label{lcyes} Let $X_{n+1}(1)$ and $X^a_{n+1}(1)$ be defined as in Definition~\ref{cdefn} and \ref{adefn} generated under fixed $\mu ^{(n+1)}(\cdot ) \in \mathcal{L}_1$ with the bounding number $L_1>1$. Then, we can find some fixed polynomial $\mathcal{P}(x )>1$ for $x>1$ such that 
	\begin{equation}  \label{note lemma 5}
	\mathbb{E}[\|X_{n+1}(1)-X_{n+1}^{a}(1)\|_{\infty}^8] \leq 
	\mathcal{P}(L_1)\Delta t_n^{8(2\alpha-\beta)}.
	\end{equation}
\end{lemma}

\subsection{Proof of Lemma~\protect\ref{orderfirstlemma}}

\begin{proof}[Proof of Lemma ~\ref{orderfirstlemma}]
	Let $X_{t}, 0\leq t \leq 1$ be the solution of the SDE under $\mu(\cdot) \in \mathcal{L}_1$ with the bounding number $L_1>1$ and $X_n(t)$ be the 
	discretization scheme in Definition~\ref{coarse}  generated under $\mu ^{(n)}(\cdot ) \in \mathcal{L}_1$ with the bounding number $L_1>1$. Additionally, let $X^\mu_n(\cdot)$  be the 
	discretization scheme modified from Definition~\ref{coarse}  generated under $\mu (\cdot ) \in \mathcal{L}_1$ with the bounding number $L_1>1$ instead of $\mu ^{n}(\cdot ) $. Then, for $0\leq t \leq 1$,
	we have the following bound on $\|X_n(t)-X_t\|_{\infty}$,
	\begin{equation}\label{eq:tri-bound}
	\|X_n(t)-X_t\|_{\infty} \leq \|X_n(t)-X_{n}^{\mu}(t)\|_{\infty}+\|X_{n}^{\mu}(t)-X_t\|_{\infty} 
	.\end{equation}
	In order to prove Lemma~\ref{orderfirstlemma}, we provide bounds for both $\|X_n(t)-X_{n}^{\mu}(t)\|_{\infty}$ and $\|X_{n}^{\mu}(t)-X_t\|_{\infty}$.
	
	For $\|X_{n}^{\mu}(t)-X_t\|_{\infty}$, using Lemma~\ref{note lemma4}, we can find a polynomial $\mathcal{P}(x )>1$ for $x>1$ such that
	\begin{equation}
	\|X_{n}^{\mu}(t)-X_t\|_{\infty} \leq  \mathcal{P}(L_1,\|B\|_{\alpha},
	\|A\|_{2\alpha},\Gamma_{\widetilde{R}})\Delta t_n^{2\alpha-\beta}
	.\end{equation}
	Similarly, using Lemma~\ref{PandC}, we can further find some polynomial $\mathcal{P}(x )>1$ for $x>1$ such that
	\begin{equation}
	\|X_{n}^{\mu}(t)-X_t\|^4_{\infty} \leq  \mathcal{P}(L_1,\|B\|_{\alpha},
	\|A\|_{2\alpha},\Gamma_{\widetilde{R}})\Delta t_n^{4(2\alpha-\beta)}
	.\end{equation}
	 It follows from Lemma~\ref{jose1} that the quantities associated with Brownian motions $\|B\|_{\alpha}$,$\|A\|_{2\alpha}$ and $\Gamma_{\widetilde{R}}$  have moments of arbitrary order. Thus, fixing $\mu(\cdot) \in \mathcal{L}_1$, we can find some polynomial  $\mathcal{P}'(x )>1$ for $x>1$ such that
	\begin{align}\label{ginal1}
	\mathbb{E}\|X^{\mu}_n(t)-X_t\|_{\infty}^4 \leq& \mathbb{E}[\mathcal{P}(L_1,\|B\|_{\alpha},
	\|A\|_{2\alpha},\Gamma_{\widetilde{R}})]\Delta t_n^{4(2\alpha-\beta)}\nonumber\\
	 \leq& \mathcal{P}'(L_1)\Delta t_n^{4(2\alpha-\beta)}\nonumber\\
	 \leq & e^{CL_1}\Delta t_n^{4(2\alpha-\beta)}
	,\end{align}
	for some constant $C>1$ since  $L_1>1$.
	Combining this with \eqref{eq:tri-bound}, we have
	\begin{equation}\label{lfp2}
	\mathbb{E}\|X_n(t)-X_t\|^4_{\infty} \leq 8 \mathbb{E}\|X_n(t)-X_{n}^{\mu}(t)\|^4_{\infty}+ 8e^{CL_1}\Delta t_n^{4(2\alpha-\beta)}.
	\end{equation}
	Thus, we can complete the proof if we can show
	\begin{equation}\label{eq:1sufficient}
	\mathbb{E}\|X_n(t)-X_{n}^{\mu}(t)\|^4_{\infty} \leq e^{CL_1} \Delta t_n^{4\alpha}
	,\end{equation}
	for some $C>1$. This is because we have, according to Lemma~\ref{alni}, 
	\begin{equation}
	4\alpha=2-4\epsilon >2-16\epsilon=4(2\alpha-\beta)
	\end{equation}
and thus \eqref{eq:1sufficient} would imply
	\begin{equation}\label{lfp1}
	\mathbb{E}\|X_n(t)-X_{n}^{\mu}(t)\|^4_{\infty} \leq e^{CL_1} \Delta t_n^{4(2\alpha-\beta)},
	\end{equation}
	since $\Delta t_n<1$. Finally, we can simply conclude the proof using \eqref{lfp2} and \eqref{lfp1} by adjusting the constant $C$.
	To prove \eqref{eq:1sufficient}, we define 
	\begin{equation}
	\bar{\mu}^{(n)}(\cdot) \triangleq\mu-\mu^{(n)}
	= \sum_{i=\lfloor 2^{n\gamma} \rfloor+1}^{\infty} \frac{\lambda_i}{i^{q}} V_i(\omega)\psi_i(\cdot)
	\end{equation}
	to be the remaining sum when we approximate $\mu$ by $\mu^{(n)}$. In Section~\ref{lookhere} of the proof of Lemma~\ref{lemma L1}, we will show that
	\begin{equation}\label{remainingfield}
	\|\bar{\mu}^{(n)}(\cdot)\|_{\infty} \leq L_1\Delta t_n^{3+\frac{q-4}{2}}.
	\end{equation}
Then, we may conduct the analysis on $\|X_n(t)-X_n^{\mu^{(n)}}(t)\|_{\infty}$ based on the following recursion: for $1 \leq i \leq d$, $0 \leq k \leq 2^n-1$,
	\begin{align}\label{first big}
	&X_{i,n}^{\mu}(t_{k+1}^{n})-X_{i,n}(t_{k+1}^n)\nonumber\\
	=&X_{i,n}^{\mu}(t_{k}^n)-X_{i,n}(t_{k}^n) 
	+\Big( \mu_i^{(n)}(X^\mu_n(t_k^n))-\mu_i^{(n)}(X_n(t_k^n))\Big)\Delta t_n 
	+\bar{\mu}^{(n)}(X_{n}^{\mu}(t_k^n))\Delta t_n  \nonumber\\
	&+\sum_{j=1}^d \sigma_{ij}(X_{n}^{\mu}(t_k^n))-\sigma_{ij}(X_n(t_k^n))\Delta B_{j,k}^n\nonumber \\
	&+\sum_{j=1}^{d^{\prime}}\sum_{l=1}^d\sum_{m=1}^{d^{\prime}}\Big( \frac{\partial \sigma_{ij}}{\partial x_l}(X_{n}^{\mu}(t_{k}^n))\sigma_{lm}(X_{n}^{\mu}(t_{k}^n))-\frac{\partial \sigma_{ij}}{\partial x_l}(X_{n}(t_{k}^n))\sigma_{lm}(X_{n}(t_{k}^n))\Big)\widetilde{A}_{mj}(t_{k}^n,t_{k+1}^n),
	\end{align}
	which is obtained by modifying \eqref{coarse} and simply taking the difference. Now, let
	\begin{equation}\label{lkkhhh}
\xi_{n,k}  \triangleq X_{n}^{\mu}(t_{k}^n)-X_{n}(t_{k}^n) 	\qquad \text{and} \qquad \xi_{i,n,k} \triangleq X_{i,n}^{\mu}(t_{k}^{n})-X_{i,n}(t_{k}^n) 
	\end{equation}
	for $1 \leq i \leq d$ and $0 \leq k \leq 2^n$ and let 
	\begin{align}\label{step order}
	\eta_{i,n,k} 
	\triangleq&\big( \mu_i^{(n)}(X^\mu_n(t_k^n))-\mu_i^{(n)}(X_n(t_k^n))\big)\Delta t_n +\bar{\mu}^{(n)}(X_{n}^{\mu}(t_k^n))\Delta t_n  \nonumber\\
	 &+\sum_{j=1}^d \Big(\sigma_{ij}(X_{n}^{\mu}(t_k^n))-\sigma_{ij}(X_n(t_k^n))\Big)\Delta B_{j,k}^n\nonumber \\
	&+\sum_{j=1}^{d^{\prime}}\sum_{l=1}^d\sum_{m=1}^{d^{\prime}}\Big( \frac{\partial \sigma_{ij}}{\partial x_l}(X_{n}^{\mu}(t_{k}^n))\sigma_{lm}(X_{n}^{\mu}(t_{k}^n))-\frac{\partial \sigma_{ij}}{\partial x_l}(X_{n}(t_{k}^n))\sigma_{lm}(X_{n}(t_{k}^n))\Big)\widetilde{A}_{mj}(t_{k}^n,t_{k+1}^n)
	,\end{align}
	so that \eqref{first big} becomes, for $1 \leq i \leq d$, $0 \leq k \leq 2^n-1$, 
	\begin{equation}\label{add4}
	\xi_{i,n,k+1} =\xi_{i,n,k} + \eta_{i,n,k}.
	\end{equation}
	Fixing ${\mu}\in\mathcal{L}_1$ and $\mu^{(n)}\in\mathcal{L}_1$ with bounding number $L_1>1$ and taking expectation on \eqref{add4} after raising it to the fourth power, we have
	\begin{align}
\mathbb{E}(\xi_{i,n,k+1} ^4)  =& \mathbb{E}(\xi_{i,n,k} ^4)+\EEb(\eta_{i,n,k} ^4)+3\mathbb{E}(\xi_{i,n,k} ^3\eta_{i,n,k} )+3\mathbb{E}(\xi_{i,n,k} \eta_{i,n,k} ^3)+6\mathbb{E}(\xi_{i,n,k} ^2\eta_{i,n,k} ^2).\label{first big 2}
	\end{align}
	It now follows from the definition of $\xi_{i,n,k}$ in \eqref{lkkhhh} that it is sufficient to show 
	\begin{equation}\label{eq:sufficient}
	\mathbb{E}\|X_n(t)-X_{n}^{\mu}(t)\|^4_{\infty}=\mathbb{E}\|\xi_{n,2^n}\|_{\infty}^4\leq e^{CL_1} \Delta t_n^{4\alpha},
	\end{equation}
	for some constant $C>1$. Thus, in what follows, we focus on the proof of \eqref{eq:sufficient}, which consists of proofs for the following two statements:  fixing ${\mu}\in\mathcal{L}_1$ and $\mu^{(n)}\in\mathcal{L}_1$ with bounding number $L_1>1$, 
	\begin{itemize}
		\item (I) We prove that there exists a constant $C>1$ and a polynomial $\mathcal{P}(x )>1$ for $x>1$ such that for $1 \leq i \leq d$ and $0 \leq k \leq 2^n$, we have
			\begin{equation}
			\mathbb{E}\abs{\xi_{i,n,k}}^4 \leq e^{CL_1\cdot t_k^n}  \Delta t_n^{4\alpha}
			\end{equation}
			if $n$ is large enough so that $2^{n} > \mathcal{P}(L_1)$.
		\item	(II) We prove that there is a polynomial $\mathcal{P}'(x )>1$ for $x>1$ such that for $1 \leq i \leq d$ and $0 \leq k \leq 2^n$, we have
		\begin{equation}
		\mathbb{E}\abs{\xi_{i,n,k}}^4\leq \mathcal{P}'(L_1)\Delta t_n^{4\alpha},
		\end{equation}
		if $n$ is not large enough and $2^{n} \leq \mathcal{P}(L_1)$.
	\end{itemize}

	\paragraph{Proof of statement (I)}
	Fixing ${\mu}\in\mathcal{L}_1$ and $\mu^{(n)}\in\mathcal{L}_1$ with bounding number $L_1>1$, we use induction on $0 \leq k \leq 2^n$. First of all, when $k=0$, for $1 \leq i \leq d$, the claim holds since $\xi_{i,n,0}= X_{i,n}^{\mu}(0)-X_{i,n}(0) =x-x=0$.
	
	Next, for $0 < k \leq 2^n-1$ , assume that the induction hypothesis holds so that whenever $0 \leq j \leq k$,
	\begin{equation}\label{inductionassumption}
	\mathbb{E}\abs{\xi^4_{i,n,j}} \leq e^{CL_1 t_j^n}\cdot \Delta t_n^{4\alpha}
	\end{equation}
	for $1 \leq i \leq d$ and some $C>1$. Our goal is to show 
	\begin{equation}
	\mathbb{E}\abs{\xi^4_{i,n,k+1}} \leq e^{CL_1 t_{k+1}^n}\cdot \Delta t_n^{4\alpha}
	\end{equation}	
	for all $1\leq i \leq d$. To do this, we provide  bounds for every term on the right hand side of \eqref{first big 2}. For $\eta^4_{i,n,k}$, according to Definition~\ref{jose0} and denoting $\bar{d}= \max\{d,d'\}$, we have
	\begin{align} \label{etabound}
	\abs{\eta_{i,n,k}}
	\leq& \|\partial{\mu_i^{(n)}}\|_{\infty}\|\xi_{n,k}\|_{\infty}\Delta t_n+\|\bar{\mu}^{(n)}\|_{\infty}\Delta t_n+\bar{d}L\|\xi_{n,k}\|_{\infty}\|B\|_{\alpha}\Delta t_n^{\alpha}\nonumber\\
	&+{\bar{d}}^3L\|\xi_{n,k}\|_{\infty}\|A\|_{2\alpha}\Delta t_n^{2\alpha} \nonumber \\
	\leq& L_1\|\xi_{n,k}\|_{\infty}\Delta t_n+L_1\Delta t_n^{4+\frac{q-4}{2}}+\bar{d}L\|\xi_{n,k}\|_{\infty}\|B\|_{\alpha}\Delta t_n^{\alpha}\nonumber\\
	&+{\bar{d}}^3L\|\xi_{n,k}\|_{\infty}\|A\|_{2\alpha}\Delta t_n^{2\alpha}
	\end{align}
	where the last line follows from \eqref{remainingfield}. Since $\xi_{n,k}$ and the shifted Brownian motion on $B(t)-B(t_k^n), t_k^n \leq t \leq t_{k+1}^n$ are independent of each other (i.e., independent increments of Brownian motion), we can consider quantities $\|B\|_{\alpha}$ and $\|A\|_{2\alpha}$ to be associated with the new Brownian motion $B(t)-B(t_k^n), t_k^n \leq t \leq t_{k+1}^n$ and thus independent of $\xi_{n,k}$. 	Consequently, it then follows from Lemma~ \ref{jose1} that we can find a constant $C' > 1$ such that
	\begin{align} \label{rhs1}
	\mathbb{E}\eta_{i,n,k} ^4  &\leq C'\bigg(L_1^4\mathbb{E}(\|\xi_{n,k}\|_{\infty}^4)\Delta t_n^{4}+L_1^4\Delta t_n^{16+2(q-4)}+\mathbb{E}(\|\xi_{n,k}\|_{\infty}^4)\Delta t_n^{4\alpha}+\mathbb{E}(\|\xi_{n,k}\|_{\infty}^4)\Delta t_n^{8\alpha}\bigg)\nonumber \\
	& \leq C''L_1^4\bar{d}^4e^{CL_1\cdot t_{k}^n}\cdot \Delta t_n^{8\alpha}
	\end{align}
	for some $C''>1$ where the last line follows from both the induction hypothesis and the fact that $8\alpha < 16+2(q-4)$ in Definition \ref{alni}. 
	
	For the bound on $\mathbb{E}(\xi_{i,n,k} ^3\eta_{i,n,k} )$ in \eqref{first big 2}, we observe the terms in \eqref{step order} and  use \eqref{remainingfield} along with the martingale property (i.e., the independence of $\Delta B_k^n$ and $X_n(t_k^n)$) to obtain
	\begin{align}\label{rhs2}
	&\mathbb{E}(\xi_{i,n,k} ^3\eta_{i,n,k} ) \nonumber \\
	=&\mathbb{E}\big[\big(X_{i,n}^{\mu}(t_{k}^n)-X_{i,n}^{\mu^{(n)}}(t_{k}^n)\big)^3 \cdot \Big(\big( \mu_i^{(n)}(X^\mu_n(t_k^n))-\mu_i^{(n)}(X^{\mu^{(n)}}_n(t_k^n))\big)\Delta t_n 
	+\bar{\mu}^{(n)}(X_{n}^{\mu}(t_k^n))\Delta t_n\Big)\big]\nonumber \\
	\leq & \mathbb{E}[L_1\|\xi_{n,k} \|_{\infty}^4\Delta t_n] +\mathbb{E}[L_1\|\xi_{n,k} \|_{\infty}^3\Delta t_n^{4+\frac{q-4}{2}}] \nonumber \\
	\leq& \bigg(\mathbb{E}(\|\xi_{n,k} \|_{\infty}^4)+\mathbb{E}(\|\xi_{n,k} \|_{\infty}^4)^{\frac{3}{4}}\Delta t_n^{\alpha}\bigg) L_1\Delta t_n 
	\leq 2L_1\bar{d}^4e^{CL_1\cdot t_k^n} \Delta t_n^{4\alpha+1}. \nonumber\\
	\end{align}
 The last inequality follows from induction hypothesis, the second inequality follows from H\"older's inequality and the fact that $\alpha < 4+\frac{q-4}{2}$ as in Definition \ref{alni}, and the first inequality follows from the bound on $\|\partial{\mu_i^{(n)}}\|_{\infty}$ in Assumption~\ref{assump:1}.  
 
 For the bound on $\mathbb{E}(\xi_{i,n,k} ^2\eta_{i,n,k} ^2 )$, using the bound on $\abs{\eta_{i,n,k}}$ in \eqref{etabound} and the fact that {$\mathbb{E}[\big(B(t)-B(s)\big)^2]=O(|t-s|)$} and $\mathbb{E}[\big(\widetilde{A}_{ij}(s,t)\big)^2]=O((t-s)^2)$ (see, for example, \cite{karatzas2012brownian}), we can find some $C' > 1$ that
	\begin{align}\label{rhs3}
	\mathbb{E}(\xi_{i,n,k} ^2\eta_{i,n,k} ^2 ) 
	\leq &C'\Big( \mathbb{E}({\|\xi_{n,k} \|_{\infty}^4}) L_1^2 \Delta t_n^2 +\mathbb{E}({\|\xi_{n,k} \|_{\infty}^2}) L_1^2 \Delta t_n^{4+q}\\
	&+ \mathbb{E}({\|\xi_{n,k} \|_{\infty}^4})  \Delta t_n+ \mathbb{E}({\|\xi_{n,k} \|_{\infty}^4}) {\Delta t_n^2}\Big) \nonumber \\
	\leq &C'\mathbb{E}(\|\xi_{n,k}\|^4_{\infty})\Delta t_n\big(L_1^2\Delta t_n+2\big)+C'\mathbb{E}(\|\xi_{n,k}\|^4_{\infty})^{\frac{1}{2}}\Delta t_n^{2\alpha+1}\big(L_1^2(\omega)\Delta t_n\big) \nonumber \\
	\leq & 2C''\bar{d}^4e^{CL_1\cdot t_k^n} (L_1^2\Delta t_n +1) \Delta t_n^{4\alpha+1}
	\end{align}
	for some $C''>1$. The last line follows from induction hypothesis. The second to last line follows from H\"older's inequality and the fact that $2\alpha+2 < 4+q$ as in Definition~\ref{alni}. 
	Finally, to bound $\mathbb{E}(\xi_{i,n,k} \eta_{i,n,k} ^3)$ in \eqref{first big 2}, following similar techniques, we use inequality \eqref{rhs1}, induction hypothesis and H\"older's inequality to obtain
	\begin{align}\label{rhs4}
	\mathbb{E}(\xi_{i,n,k} \eta_{i,n,k} ^3)  
	\leq& (\mathbb{E}(\xi_{i,n,k} ^4))^{\frac{1}{4}}(\mathbb{E}(\eta_{i,n,k} ^4))^{\frac{3}{4}} 
	\leq C''L_1^3\bar{d}^3e^{CL_1\cdot t_k^n} \Delta t_n^{7\alpha}.
	\end{align}
	Now we are ready to prove the induction hypothesis. Let
	\begin{equation}\label{push1}  C=12C''\bar{d}^4+6\bar{d}^4+1 \quad \text{and} \quad \mathcal{P}(x)=\bigg(C''(x^4\bar{d}^4+3x^3\bar{d}^3+12x^2\bar{d}^4)\bigg)^3.
	\end{equation}
	It is easy to check that $C>1$ and the polynomial $\mathcal{P}(x )>1$ for $x>1$. Then, it follows from Definition~\ref{alni} and standard calculation that if $n$ is large enough that $2^n > \mathcal{P}(L_1)$(i.e., $\Delta t_n < (\mathcal{P}(L_1))^{-1}$), then
	\begin{align}\label{push2}
	&C''L_1^4\bar{d}^4 \Delta t_n^{4\alpha-1}+   3C''L_1^3\bar{d}^3 \Delta t_n^{3\alpha-1}+12C''L_1^2\bar{d}^4\Delta t_n  \nonumber\\
	\leq & (C''L_1^4\bar{d}^4 +   3C''L_1^3\bar{d}^3 +12C''L_1^2\bar{d}^4)\Delta t_n^{3\alpha-1}\nonumber\\
	=& (\mathcal{P}(L_1))^{\frac{1}{3}}\Delta t_n^{3\alpha-1}
	< (\mathcal{P}(L_1))^{\frac{4}{3}-3\alpha}<1
	\end{align}
	where the last inequality follows from the fact that $0>\frac{4}{3}-3\alpha$, $L_1>1$ and $\mathcal{P}(x)>1$ for $x>1$. Thus, for $n$ such that $2^n > \mathcal{P}(L_1)$,  we use \eqref{first big 2}, H\"older's inequality and the bound acquired in \eqref{rhs1}, \eqref{rhs2}, \eqref{rhs3} and \eqref{rhs4} to get
	\begin{align}
	\mathbb{E}(\xi_{i,n,k+1} ^4) 
	=& \mathbb{E}(\xi_{i,n,k} ^4)+\EEb(\eta_{i,n,k} ^4)+3\mathbb{E}(\xi_{i,n,k} ^3\eta_{i,n,k} )+3\mathbb{E}(\xi_{i,n,k} \eta_{i,n,k} ^3)+6\mathbb{E}(\xi_{i,n,k} ^2\eta_{i,n,k} ^2)\nonumber \\
	\leq &e^{CL_1\cdot t_k^n}\Delta t_n^{4\alpha}   \big(1+ C''L_1^4\bar{d}^4 \Delta t_n^{4\alpha}+ 6L_1\bar{d}^4 \Delta t_n+  3C''L_1^3\bar{d}^3 \Delta t_n^{3\alpha}+12C''\bar{d}^4 (L_1^2\Delta t_n +1) \Delta t_n	\big) \nonumber \\
	\leq & e^{CL_1 t_k^n}\Delta t_n^{4\alpha}  \big(1+(6\bar{d}^4+12C''\bar{d}^4+1)L_1\Delta t_n\big) \nonumber \\
	=&e^{CL_1 t_k^n}\Delta t_n^{4\alpha}  (1+CL_1\Delta t_n) \leq e^{CL_1 t_{k+1}^n}\Delta t_n^{4\alpha}
	\end{align}
	where the last line follows from convexity of exponential function: $e^y \geq e^x + e^x\cdot(y-x)$ for $y\geq x$. The second to last inequality follows from \eqref{push1}, \eqref{push2} and the fact that $L_1>1$. This concludes the induction. However, since $t_n^k \leq 1$ for all $0 \leq k \leq 2^n$, we have actually proven that when $2^n > \mathcal{P}(L_1)$(i.e., $\Delta t_n < (\mathcal{P}(L_1))^{-1}$),
	\begin{equation}\label{ginal2}
	\mathbb{E}\|X^{\mu}_{i,n}(t)-X_{i,n}(t)\|_{\infty}^4 \leq e^{CL_1} \cdot \Delta t_n^{4\alpha}
	\end{equation}
	for all $1 \leq i \leq d$ and $0\leq t \leq 1$.
	\paragraph{Proof of statement (II)}
	Next, we extend the result to the case where $2^n \leq \mathcal{P}(L_1)$. By observing \eqref{first big}, we can find polynomial function $\mathcal{P}'(x) > 1$ for $x>1$ so that:
	\begin{equation}
	\abs {({X_{i,n}^{\mu}(t_{k+1}^{n})-X_{i,n}(t_{k+1}^n)})-({X_{i,n}^{\mu}(t_{k}^{n})-X_{i,n}(t_{k}^n)})} \leq \mathcal{P}'(L_1,\|B\|_{\alpha},\|\widetilde{A}\|_{2\alpha})\Delta t_n^{\alpha}.
	\end{equation} 
	Since the number of iterations in the discretization scheme $2^n$ is at most ${\mathcal{P}(L_1)}$, we have
	\begin{equation}
	\|X_{i,n}^{\mu}(\cdot)-X_{i,n}(\cdot)\|_{\infty} \leq {\mathcal{P}(L_1)}\mathcal{P}'(L_1,\|B\|_{\alpha},\|\widetilde{A}\|_{2\alpha})  \Delta t_n^{\alpha},
	\end{equation}
	and consequently, from Lemma~\ref{jose1}, that
	\begin{align}\label{ginal3}
	\mathbb{E}\|X_{i,n}^{\mu}(\cdot)-X_{i,n}(\cdot)\|_{\infty}^4 \leq\mathcal{P}''(L_1)\Delta t_n^{4\alpha}
	\end{align}
	for some polynomial  $\mathcal{P}''(x)> 1$ when $x>1$. 
	
	This concludes the proof of Lemma~\ref{orderfirstlemma}.
\end{proof}

\subsection{Proof of Lemma~\protect\ref{finitefourthmoment}}

We first prove the second claim \eqref{ffm2} of Lemma~\ref{finitefourthmoment}: 
\begin{proof}[Proof of \eqref{ffm2} in \lemref{finitefourthmoment}]
Assume without loss of generality that $x=0$. Fixing ${\mu}\in\mathcal{L}_1$ and $\mu^{(n)}\in\mathcal{L}_1$ with bounding number $L_1>1$, since
\begin{align}
\sup_{1\leq k \leq 2^n}\abs{X_{i,n}(t_{k}^n)}^4\leq & \Big(\sum_{k=0}^{2^n}\|\mu\|_{\infty}\Delta t_n+\sum_{j}\sup_{1\leq h \leq 2^n} |\sum_{k=0}^{h}\sigma_{ij}(X_n(t_{k}^n))\Delta B_{j,k}^n| \nonumber \\
&+\sum_{j,l,m} \sup_{1\leq h \leq 2^n} \abs{\sum_{k=0}^{h}\frac{\partial \sigma_{ij}}{\partial x_l}(X_{n}(t_{k}^n))\sigma_{lm}(X_{n}(t_{k}^n))\widetilde{A}_{mj}(t_{k}^n,t_{k+1}^n)}\Big),
\end{align}
it follows from \eqref{coarse} and Assumption \ref{assump:1} that we can find constant $C>1$ such that
	\begin{align}
	\sup_{1\leq k \leq 2^n}\abs{X_{i,n}(t_{k}^n)}^4\leq &C \cdot \Big(L_1^4+\sum_{j}\big(\sup_{1\leq h \leq 2^n} |\sum_{k=0}^{h}\sigma_{ij}(X_n(t_{k}^n))\Delta B_{j,k}^n|\big)^4 \nonumber \\
	&+\sum_{j,l,m} \big(\sup_{1\leq h \leq 2^n} \abs{\sum_{k=0}^{h}\frac{\partial \sigma_{ij}}{\partial x_l}(X_{n}(t_{k}^n))\sigma_{lm}(X_{n}(t_{k}^n))\widetilde{A}_{mj}(t_{k}^n,t_{k+1}^n)}\big)^4\Big).
	\end{align}
	Now, using the fact that $\mathbb{E}\big(B(t)-B(s)\big)^4=O(t-s)^2$ and $\mathbb{E}\big(\widetilde{A}_{ij}(s,t)\big)^4=O(t-s)^4$, we recall Burkholder-Davis-Gundy inequality \cite{burkholder1972} to further find constant $C'>1$ and $C''>1$ so that
	\begin{align}
		\mathbb{E}\sup_{1\leq k \leq 2^n}\abs{X_{i,n}(t_{k}^n)}^4\leq &C \cdot \Big(L_1^4+\sum_{j}\sup_{1\leq h \leq 2^n} |\sum_{k=0}^{h}\sigma_{ij}(X_n(t_{k}^n))\Delta B_{j,k}^n|^4 \nonumber \\
		&+\sum_{j,l,m} \sup_{1\leq h \leq 2^n} \abs{\sum_{k=0}^{h}\frac{\partial \sigma_{ij}}{\partial x_l}(X_{n}(t_{k}^n))\sigma_{lm}(X_{n}(t_{k}^n))\widetilde{A}_{mj}(t_{k}^n,t_{k+1}^n)}^4\Big)\nonumber\\
	\leq& C^{\prime}\Big(L_1^4+\sum_{j}	\mathbb{E}\big( \sum_{k=0}^{2^n}(\sigma_{ij}(X_n(t_{k}^n)))^2(\Delta B_{j,k}^n)^2\big)^2 \nonumber \\
	&+ \sum_{j,l,m} \mathbb{E}\big(\sum_{k=0}^{2^n} (\frac{\partial \sigma_{ij}}{\partial x_l}(X_{n}(t_{k}^n)))^2\sigma_{lm}^2(X_{n}(t_{k}^n))(\widetilde{A}_{mj}(t_{k}^n,t_{k+1}^n))^2\big)^2\Big) \nonumber \\
	\leq& C^{\prime}\Big(L_1^4+\sum_{j=1}^{d^{\prime}}\mathbb{E} 2^n\sum_{k=0}^{2^n}(\sigma_{ij}(X_n(t_{k}^n)))^4(\Delta B_{j,k}^n)^4\nonumber \\
	&+ \sum_{j,l,m} \mathbb{E}2^n\sum_{k=0}^{2^n} (\frac{\partial \sigma_{ij}}{\partial x_l}(X_{n}(t_{k}^n)))^4\sigma_{lm}^4(X_{n}(t_{k}^n))(\widetilde{A}_{mj}(t_{k}^n,t_{k+1}^n))^4\Big) \nonumber \\
	\leq& C^{\prime\prime}(L_1^4+L^42^n(\sum_{k=0}^{2^n}\Delta t_n^2 +\sum_{k=0}^{2^n}\Delta t_n^4))<\mathcal{P}(L_1)
	\end{align}
	for some polynomial function $\mathcal{P}(x) > 1$ for $x>1$. Now, the claim on $\mathbb{E} |f(X_{n_0}(1))|^4$ follows by invoking the bound on 	$\|\frac{ \partial f}{\partial x_{i}}\|_{\infty} $ in Assumption ~\ref{assump:2}.
\end{proof}
We proceed to the proof of the first claim of Lemma~\ref{finitefourthmoment}, namely %
\eqref{ffm1}. 
\begin{proof}[Proof of \eqref{ffm1} in \lemref{finitefourthmoment}]
It follows from Equation \eqref{elementary inequality} in \cite{giles2014} that we have 
for $p \geq 2$ 
\begin{equation}\label{gileslemma}
\lvert \Delta_n\rvert^p \leq 2^{p-1}L^p\mathbb{E}[\|\frac{1}{2}
(X_{n+1}(1)+X_{n+1}^{a}(1))-X_n(1)\|_{\infty}^p]+2^{-p-1}L^p\mathbb{E}%
[\|X_{n+1}(1)-X_{n+1}^{a}(1)\|_{\infty}^{2p}].
\end{equation}
Thus, according to \eqref{gileslemma}, in order to  prove Lemma~\ref
{finitefourthmoment}, 
\begin{equation}
\mathbb{E}(\Delta_n^4) \leq e^{CL_1} \Delta
t_n^{4-\delta},
\end{equation}
where $\delta>0$ is defined in Definition~\ref{alni}, it is sufficient to provide an upper bound on
\begin{equation}  \label{mmmm1}
\mathbb{E}[\|\frac{1}{2} (X_{n+1}(1)+X_{n+1}^{a}(1))-X_n(1)%
\|_{\infty}^4] 
\end{equation}
and 
\begin{equation}  \label{mmmm2}
\mathbb{E}[\|X_{n+1}(1)-X_{n+1}^{a}(1)\|_{\infty}^8]. 
\end{equation}
 Note
that bound on \eqref{mmmm2} is provided by Lemma~\ref{lcyes} since ${4-\delta} < {8(2\alpha-\beta)}$ as in Definition~\ref{alni} and $\mathcal{P}(L_1)<e^{CL_1}$ for appropriately chosen $C>1$. So we just need to prove %
\eqref{mmmm1}.
First we write the recursion for $X_{n+1}(\cdot)$ over the coarse step $%
\Delta t_n $ instead of $\Delta t_{n+1}$. For $1\leq i \leq d$ and $%
1\leq k \leq 2^{n}$, adding up two steps of recursion for $X_{n+1}(\cdot)$,
we have: 
\begin{align}
X_{i,n+1}(t_{k+1}^n)
=X&_{i,n+1}(t_k^n)+\mu_i^{(n+1)}(X_{n+1}^f(t_k^n))+\sum_{j=1}^{d^{\prime}}
\sigma_{ij}(X_{n+1}(t_{k}^n))\Delta B_{j,k}^n
\notag \\
&+
\sum_{j=1}^{d^{\prime}}\sum_{l=1}^d\sum_{m=1}^{d^{\prime}} \frac{\partial \sigma_{ij}}{\partial x_l%
}(X_{n+1}(t_{k}^n))\sigma_{lm}(X_{n+1}(t_{k}^n))\widetilde{A}%
_{mj}(t_{k}^n,t_{k+1}^n)  \notag \\
&- \sum_{j=1}^{d^{\prime}}\sum_{l=1}^d\sum_{m=1}^{d^{\prime}} \frac{\partial \sigma_{ij}}{\partial
	x_l}(X_{n+1}(t_{k}^n))\sigma_{lm}(X_{n+1}(t_{k}^n))(\Delta
B_{j,2k}^{n+1}\Delta B_{m,2k+1}^{n+1}-\Delta B_{m,2k}^{n+1}\Delta
B_{j,2k+1}^{n+1}) \notag\\
&+N_{i,n,k}^f+M_{i,n,k}^{f,(1)}+ M_{i,n,k}^{f,(2)} + M_{i,n,k}^{f,(3)}
\end{align}
where we define 
\begin{align}
M_{i,n,k}^{f,(2)} \triangleq& \Big(\sum_{j=1}^{d^{\prime}}(%
\sigma_{ij}(X_{n+1}(t^{n+1}_{2k+1}))-\sigma_{ij}(X_{n+1}(t^{n}_{k})))-%
\sum_{j=1}^{d^{\prime}}\sum_{l=1}^d\sum_{m=1}^{d^{\prime}} \big(\frac{\partial \sigma_{ij}}{%
	\partial x_l}\cdot\sigma_{lm}\big)(X_{n+1}(t_{k}^n))\Delta B_{m,2k}^{n+1}%
\Big)\Delta B_{j,2k+1}^{n+1}\nonumber\\
M_{i,n,k}^{f,(3)}\triangleq& \sum_{j=1}^{d^{\prime}}\sum_{l=1}^d\sum_{m=1}^{d^{\prime}} \Big(\big(%
\frac{\partial \sigma_{ij}}{\partial x_l}\cdot\sigma_{lm}\big)%
(X_{n+1}(t^{n+1}_{2k+1}))-\big(\frac{\partial \sigma_{ij}}{\partial x_l}%
\cdot\sigma_{lm}\big)(X_{n+1}(t_{k}^n))\Big)\widetilde{A}%
_{mj}(t_{2k+1}^{n+1},t_{k+1}^n)\nonumber\\
M_{i,n,k}^{f,(1)} \triangleq& \big(\sum_{j=1}^d \frac{\partial \mu_i^{(n+1)}}{%
	\partial x_j}(X_{n+1}(t_{k}^n))\sum_{m=1}^{d^{\prime}}
\sigma_{jm}(X_{n+1}(t_{k}^n))\Delta B_{m,2k}^{n+1}\big) \frac{\Delta t_n}{2}
\end{align}
and 
\begin{align}  \label{Nf}
N_{i,n,k}^f \triangleq& \big({\mu}^{(n+1)}_i(X_{n+1}(t^{n+1}_{2k+1}))-{\mu}%
^{(n+1)}_i(X_{n+1}(t^{n}_{k}))\big)\frac{\Delta t_n}{2}-M_{i,n,k}^{f,(1)} 
\notag \\
=&\Big(\sum_{j=1}^d \frac{\partial \mu_i^{(n+1)}}{\partial x_j}%
(X_{n+1}(t_{k}^n))(X_{j,n+1}(t_{2k+1}^{n+1})-X_{j,n+1}(t_{k}^{n}))  \notag \\
&+\frac{1}{2} \sum_{j=1}^d \sum_{m=1}^d \frac{\partial ^2 \mu_i^{(n+1)}}{%
	\partial x_j \partial x_m}%
(\eta)(X_{j,n+1}(t_{2k+1}^{n+1})-X_{j,n+1}(t_{k}^{n}))
(X_{m,n+1}(t_{2k+1}^{n+1})-X_{m,n+1}(t_{k}^{n}))\Big)\frac{\Delta t_n}{2}%
-M_{i,n,k}^{f,(1)}  \notag \\
=&\bigg(\sum_{j=1}^d \frac{\partial \mu_i^{(n+1)}}{\partial x_j}%
(X_{n+1}(t_{k}^n))\Big( {\mu}_j^{(n+1)}(X_{n+1}(t_k^n))\frac{\Delta t_n}{2}%
+\sum_{m,l,\widetilde{m}} \big(\frac{\partial \sigma_{jm}}{\partial x_l}%
\cdot\sigma_{l\widetilde{m}}\big)(X_{n+1}(t_{k}^n))\widetilde{A}_{m%
	\widetilde{m}}(t_k^n,t_{2k+1}^{n+1})\Big)  \notag \\
+&\frac{1}{2} \sum_{j,m=1}^d \frac{\partial ^2 \mu_i^{(n+1)}}{\partial x_j
	\partial x_m}(\rho)(X_{j,n+1}(t_{2k+1}^{n+1})-X_{j,n+1}(t_{k}^{n}))
(X_{m,n+1}(t_{2k+1}^{n+1})-X_{m,n+1}(t_{k}^{n}))\bigg)\frac{\Delta t_n}{2}
\end{align}
for some $\rho$ that lies between $X_{n+1}(t_k^n)$ and $X_{n+1}(t_{2k+1}^{n+1})$.

Furthermore, we similarly define $N_{i,n,k}^a, M_{i,n,k}^{a,(\cdot)}$ associated with $%
X_{n+1}^a(\cdot) $, $ B^{n+1,a}(t)$ and $\widetilde{A}%
^a(t^{n+1}_{k+1},t^{n+1}_k)$ so we can write the recursion over the coarse
step $\Delta t_n$ for $X_{n+1}^a(\cdot)$ by using \eqref{trihere} in Definition~\ref{adefn}: 
\begin{align}
X_{i,n+1}^a(t_{k+1}^n)  
=&X_{i,n+1}^a(t_k^n)+\mu_i^{(n+1)}(X_{n+1}^a(t_k^n))+\sum_{j=1}^{d^{\prime}}
\sigma_{ij}(X_{n+1}^a(t_{k}^n))\Delta
B_{j,k}^n \notag \\
&+
\sum_{j=1}^{d^{\prime}}\sum_{l=1}^d\sum_{m=1}^{d^{\prime}} \frac{\partial \sigma_{ij}}{\partial x_l%
}(X^a_{n+1}(t_{k}^n))\sigma_{lm}(X^a_{n+1}(t_{k}^n))\widetilde{A}^a%
_{mj}(t_{k}^n,t_{k+1}^n)  \notag \\
&- \sum_{j=1}^{d^{\prime}}\sum_{l=1}^d\sum_{m=1}^{d^{\prime}} \frac{\partial \sigma_{ij}}{\partial
	x_l}(X^a_{n+1}(t_{k}^n))\sigma_{lm}(X^a_{n+1}(t_{k}^n))(\Delta
B_{j,2k}^{n+1,a}\Delta B_{m,2k+1}^{n+1,a}-\Delta B_{m,2k}^{n+1,a}\Delta
B_{j,2k+1}^{n+1,a}) \notag\\
&+N_{i,n,k}^a+M_{i,n,k}^{a,(1)}+ M_{i,n,k}^{a,(2)} +
M_{i,n,k}^{a,(3)}.
\end{align}
Now, combining these results, we can write the recursion for $\bar{X}_{n+1}(\cdot)\triangleq \frac{1}{2}
\big(X_{n+1}(\cdot) + X_{n+1}^a(\cdot)\big)$ over the coarse step $\Delta
t_n $: 
\begin{align}  \label{barbar}
\bar{X}_{i,n+1}(t_{k+1}^n) =&\bar{X}_{i,n+1}(t_k^n) + \mu_i^{(n+1)}(\bar{X}%
_{n+1}(t_k^n))\Delta t_n + \sum_{j=1}^{d^{\prime}} \sigma_{ij}(\bar{X}%
_{n+1}(t_k^n))\Delta B_{j,k}^n  \notag \\
&+\sum_{j=1}^{d^{\prime}}\sum_{l=1}^d\sum_{m=1}^{d^{\prime}} \frac{\partial \sigma_{ij}}{\partial
	x_l}(\bar{X}_{n+1}(t_k^n))\sigma_{lm}(\bar{X}_{n+1}(t_k^n)) \widetilde{A}%
_{mj}(t_k^n,t_{k+1}^n)+R_{i,n,k},
\end{align}
where we define 
\begin{align} \label{REINK1}
R_{i,n,k} \triangleq&
N^{(1)}_{i,n,k}+M^{(1)}_{i,n,k}+M^{(2)}_{i,n,k}+M^{(3)}_{i,n,k}  \notag \\
+&\frac{1}{2}(N_{i,n,k}^f+M_{i,n,k}^{f,(1)}+ M_{i,n,k}^{f,(2)} +
M_{i,n,k}^{f,(3)}+N_{i,n,k}^a+M_{i,n,k}^{a,(1)}+ M_{i,n,k}^{a,(2)} +
M_{i,n,k}^{a,(3)}),
\end{align}
where we define
\begin{align}  \label{REINK}
N_{i,n,k}^{(1)} \triangleq& \frac{1}{2}\Big(\mu_i^{(n+1)}(X_{n+1}(t_k^n))+%
\mu_i^{(n+1)}(X_{n+1}^a(t_k^n))\Big)-\mu_i^{(n+1)}(\bar{X}_{n+1}(t_k^n)), 
\notag \\
M_{i,n,k}^{(1)}\triangleq& \sum_{j=1}^{d^{\prime}}\Big(\frac{1}{2}\big(%
\sigma_{ij}(X_{n+1}(t_k^n))+\sigma_{ij}(X_{n+1}^a(t_k^n))\big)-\sigma_{ij}(%
\bar{X}_{n+1}(t_k^n))\Big)\Delta B_{j,k}^n,  \notag \\
M_{i,n,k}^{(2)}\triangleq& \sum_{j,m=1}^{d^{\prime}}\sum_{l=1}^d \Big(\frac{1}{2}\big((%
\frac{\partial \sigma_{ij}}{\partial x_l}\cdot\sigma_{lm})(X_{n+1}(t_k^n))+(%
\frac{\partial \sigma_{ij}}{\partial x_l}\cdot\sigma_{lm})(X_{n+1}^a(t_k^n))%
\big)  \notag \\
&-\big(\frac{\partial \sigma_{ij}}{\partial x_l}\cdot\sigma_{lm}\big)(\bar{X}%
_{n+1}(t_k^n))\Big)\widetilde{A}_{mj}(t_k^n,t_{k+1}^n) ,  \notag \\
M_{i,n,k}^{(3)}\triangleq& \sum_{j,m=1}^{d^{\prime}}\sum_{l=1}^d \frac{1}{2}\big((\frac{%
	\partial \sigma_{ij}}{\partial x_l}\cdot\sigma_{lm})(X_{n+1}(t_k^n))-(\frac{%
	\partial \sigma_{ij}}{\partial x_l}\cdot\sigma_{lm})(X_{n+1}^a(t_k^n))\big) 
\notag \\
&\cdot(\Delta B_{j,2k}^{n+1}\Delta B_{m,2k+1}^{n+1}-\Delta
B_{m,2k}^{n+1}\Delta B_{j,2k+1}^{n+1}).
\end{align}

Finally, subtract the recursion in \eqref{coarse} for $X_n(\cdot)$ from $\bar{X}_n(\cdot)$ to obtain
\begin{align}  \label{barc}
&\bar{X}_{i,n+1} (t_{k+1}^n)-X_{i,n}(t_{k+1}^n)\nonumber\\ =&\bar{X}_{i,n+1}
(t_{k}^n)-X_{i,n}(t_{k}^n)+ (\mu^{(n)}_i(\bar{X}_{n+1}
(t_{k}^n))-\mu^{(n)}_i(X_{n}(t_{k}^n)))\Delta t_n  \notag \\
&+({\mu}^{(n+1)}_i-{\mu}^{(n)}_i)(\bar{X}_{i,n+1} (t_{k}^n))\Delta
t_n+ \sum_{j=1}^{d^{\prime}} \big(\sigma_{ij}(\bar{X}_{i,n+1}
(t_{k}^n))-\sigma_{ij}(X_{i,n}(t_{k}^n))\big)\Delta B_{j,k}^n  \notag \\
&+\sum_{j=1}^{d^{\prime}}\sum_{l=1}^d\sum_{m=1}^{d^{\prime}} \big((\frac{\partial \sigma_{ij}}{%
	\partial x_l}\cdot\sigma_{lm})(\bar{X}_{i,n+1} (t_{k}^n))-(\frac{\partial
	\sigma_{ij}}{\partial x_l}\cdot\sigma_{lm})(X_{i,n}(t_{k}^n))\big)\widetilde{%
	A}_{mj}(t_{k}^n,t_{k+1}^n) +R_{i,n,k}
\end{align}
We are now ready to prove \eqref{mmmm1} by bounding  $%
\mathbb{E}\lvert \bar{X}_{i,n+1}
(t_{k}^n)-X_{i,n}(t_{k}^n)\rvert^4$. Similarly as in the proof of Lemma~\ref{orderfirstlemma}, we simplify the notation by defining 
\begin{equation}
\xi_{i,n,k} \triangleq \bar{X}_{i,n+1} (t_{k}^n)-X_{i,n}(t_{k}^n) \qquad 
\text{and} \qquad \xi_{n,k} \triangleq \bar{X}_{n+1} (t_{k}^n)-X_{n}(t_{k}^n)
\end{equation}
with 
\begin{align}  \label{etaagain}
\eta_{i,n,k} \triangleq & (\mu^{(n)}_i(\bar{X}_{n+1}
(t_{k}^n))-\mu^{(n)}_i(X_{n}(t_{k}^n)))\Delta t_n + \sum_{j=1}^{d^{\prime}} \big(%
\sigma_{ij}(\bar{X}_{i,n+1} (t_{k}^n))-\sigma_{ij}(X_{i,n}(t_{k}^n))\big)%
\Delta B_{j,k}^n  \notag \\
&+\sum_{j=1}^{d^{\prime}}\sum_{l=1}^d\sum_{m=1}^{d^{\prime}} \big((\frac{\partial \sigma_{ij}}{%
	\partial x_l}\cdot\sigma_{lm})(\bar{X}_{i,n+1} (t_{k}^n))-(\frac{\partial
	\sigma_{ij}}{\partial x_l}\cdot\sigma_{lm})(X_{i,n}(t_{k}^n))\big)\widetilde{%
	A}_{mj}(t_{k}^n,t_{k+1}^n)  \notag \\
&+R_{i,n,k}+({\mu}^{(n+1)}_i-{\mu}^{(n)}_i)(\bar{X}_{i,n+1} (t_{k}^n))\Delta
t_n,
\end{align}
so that we have
\begin{equation}
\xi_{i,n,k+1}=\xi_{i,n,k}+\eta_{i,n,k}
\end{equation}
for $0 \leq k \leq 2^n-1$. Fixing $V_1,V_2,...$ such that   $\mu\in\mathcal{L}_1$ and $\{{\mu}^{(n)}\}_{n\geq 0}\subset\mathcal{L}_1$, we want to find constant $C>1$ and polynomial $%
\mathcal{P}(x) > 1$ for $x>1$ such that if $n$ is large enough that $2^{n} > 
\mathcal{P}(L_1)$, then
\begin{equation}  \label{inductionhypothesis}
\mathbb{E}(\xi_{i,n,k})^4 \leq e^{CL_1\cdot
	t_k^n}\Delta t_n^{4-\delta}
\end{equation}
 for all $1 \leq i \leq d$ and $0 \leq k \leq
 2^n$. Similarly, we prove by induction on $0 \leq k \leq 2^n$. We first need to analyze all
the terms of $R_{i,n,k}$ in \eqref{REINK1}. 

We start by bounding $N^{(1)}_{i,n,k}$ in \eqref{REINK}
using Taylor expansion
\begin{align}
N_{i,n,k}^{(1)} \triangleq& \frac{1}{2}\Big(\mu_i^{(n+1)}(X_{n+1}(t_k^n))-%
\mu_i^{(n+1)}(X_{n+1}^a(t_k^n))\Big)-\mu_i^{(n+1)}(\bar{X}_{n+1}(t_k^n)) 
\notag \\
=&\frac{1}{16} \sum_{j=1}^d\sum_{m=1}^d \Big(\frac{\partial ^2 \mu_i}{%
	\partial x_j \partial x_m}(\rho_1)+\frac{\partial ^2 \mu_i}{\partial x_j
	\partial x_m}(\rho_1^{^{\prime }})\Big)%
(X_{j,n+1}(t_k^n)-X^a_{j,n+1}(t_k^n))(X_{m,n+1}(t_k^n)-X^a_{m,n+1}(t_k^n))%
\Delta t_n
\end{align}
where $\rho_1$ and $\rho_1^{\prime}$ lie somewhere between $X^a_{n+1}(t^n_k)$
and $X_{n+1}(t^n_k)$. Now we use Lemma~\ref{note lemma 5} on $%
(X_{j,n+1}(t_k^n)-X^a_{j,n+1}(t_k^n))(X_{m,n+1}(t_k^n)-X^a_{m,n+1}(t_k^n))$
and  H\"older's inequality to obtain
\begin{align}  \label{finalfinal1}
\mathbb{E}(N_{i,n,k}^{(1)})^4 < \mathcal{P}(L_1)\Delta t_n^{8(2\alpha-\beta)+4},
	\end{align}
	for some fixed polynomial $\mathcal{P}(x)>1$ for $x>1$. 
	
	Now, for $N_{i,n,k}^f$ in \eqref{Nf}, we also use Taylor expansion to obtain
	\begin{align}
	N^f_{i,n,k} =&\bigg(\sum_{j=1}^d \frac{\partial \mu_i^{(n+1)}}{\partial x_j}%
	(X_{n+1}(t_{k}^n))\Big( {\mu}_j^{(n+1)}(X_{n+1}(t_k^n))\frac{\Delta t_n}{2} 
	\notag \\
	&+\sum_{m=1}^{d^{\prime}}\sum_{l=1}^d\sum_{\widetilde{m}=1}^{d^{\prime}} \big(\frac{\partial
		\sigma_{jm}}{\partial x_l}\cdot\sigma_{l\widetilde{m}}\big)(X_{n+1}(t_{k}^n))%
	\widetilde{A}_{m\widetilde{m}}(t_k^n,t_{2k+1}^{n+1})\Big)  \notag \\
	&+\frac{1}{2} \sum_{j=1}^d \sum_{m=1}^d \frac{\partial ^2 \mu_i^{(n+1)}}{%
		\partial x_j \partial x_m}%
	(\rho)(X_{j,n+1}(t_{2k+1}^{n+1})-X_{j,n+1}(t_{k}^{n}))
	(X_{m,n+1}(t_{2k+1}^{n+1})-X_{m,n+1}(t_{k}^{n}))\bigg)\frac{\Delta t_n}{2},
	\end{align}
	by using Lemma~\ref{note lemma3} on $%
	(X_{j,n+1}(t_{2k+1}^{n+1})-X_{j,n+1}(t_{k}^{n}))(X_{m,n+1}(t_{2k+1}^{n+1})-X_{m,n+1}(t_{k}^{n})) 
	$ and Lemma~\ref{jose1}. Thus, we can also find some fixed polynomial $\mathcal{P}%
	(x) > 1$ for $x>1$ such that 
	\begin{align}  \label{finalfinal2}
	\mathbb{E}(N^f_{i,n,k})^4 < \mathcal{P}(L_1)\Delta t_n^{8\alpha+4}.
	\end{align}
	For other terms of $R_{i,n,k}$ in \eqref{REINK}, we similarly write out their Taylor
	expansion as follows: 
	\begin{align}
	M_{i,n,k}^{(1)} \triangleq& \sum_{j=1}^{d^{\prime}}\Big(\frac{1}{2}%
	(\sigma_{ij}(X_{n+1}(t_k^n))+\sigma_{ij}(X_{n+1}^a(t_k^n)))-\sigma_{ij}(\bar{%
		X}_{n+1}(t_k^n))\Big)\Delta B_{j,k}^n  \notag \\
	 =&\frac{1}{16}\sum_{j=1}^{d^{\prime}}\sum_{m,l=1}^d \Big(\frac{\partial ^2 \sigma_{ij}%
	}{\partial x_m \partial x_l}(\rho_2)+\frac{\partial ^2 \sigma_{ij}}{\partial
	x_m \partial x_l}(\rho_2^{^{\prime }})\Big)%
(X_{m,n+1}(t_k^n)-X^a_{m,n+1}(t_k^n))  \notag \\
&\cdot(X_{l,n+1}(t_k^n)-X^a_{l,n+1}(t_k^n))\Delta B_{j,k}^n, \nonumber\\
\end{align}
and also
\begin{align}
M_{i,n,k}^{(2)}\triangleq& \sum_{j,m=1}^{d^{\prime}}\sum_{l=1}^d \Big(\frac{1}{2}\big((%
\frac{\partial \sigma_{ij}}{\partial x_l}\cdot\sigma_{lm})(X_{n+1}(t_k^n))+(%
\frac{\partial \sigma_{ij}}{\partial x_l}\cdot\sigma_{lm})(X_{n+1}^a(t_k^n))%
\big)-(\frac{\partial \sigma_{ij}}{\partial x_l}\cdot\sigma_{lm})(\bar{X}%
_{n+1}(t_k^n))\Big)  \notag \\
& \cdot\widetilde{A}_{mj}(t_k^n,t_{k+1}^n)  \notag \\
=&\frac{1}{4}\sum_{j,m=1}^{d^{\prime}}\sum_{l,l^{^{\prime }}=1}^d\Big((\frac{\partial^2
	\sigma_{ij}}{\partial x_l \partial x_{l^{^{\prime }}}}\sigma_{lm}+\frac{%
	\partial \sigma_{ij}}{\partial x_l}\frac{\partial \sigma_{lm}}{\partial
	x_{l^{^{\prime }}}})(\rho_3)-(\frac{\partial^2 \sigma_{ij}}{\partial x_l
	\partial x_{l^{^{\prime }}}}\sigma_{lm}+\frac{\partial \sigma_{ij}}{\partial
	x_l}\frac{\partial \sigma_{lm}}{\partial x_{l^{^{\prime }}}})(\rho_4)\Big) 
\notag \\
&\cdot (X_{l^{^{\prime }},n+1}(t_k^n)-X_{l^{^{\prime
		}},n+1}^a(t_k^n))\widetilde{A}_{mj}(t_k^n,t_{k+1}^n) \nonumber\\
M_{i,n,k}^{(3)}\triangleq& \sum_{j,m,l} \frac{1}{4}\big((\frac{\partial
		\sigma_{ij}}{\partial x_l}\cdot\sigma_{lm})(X_{n+1}(t_k^n))-(\frac{\partial
		\sigma_{ij}}{\partial x_l}\cdot\sigma_{lm})(X_{n+1}^a(t_k^n))\big)(\Delta
	B_{j,2k}^{n+1}\Delta B_{m,2k+1}^{n+1}-\Delta B_{m,2k}^{n+1}\Delta
	B_{j,2k+1}^{n+1})  \notag \\
	=&\frac{1}{8}\sum_{j,m=1}^{d^{\prime}}\sum_{l,l^{^{\prime }}=1}^d \Big((\frac{%
		\partial^2 \sigma_{ij}}{\partial x_l \partial x_{l^{^{\prime }}}}\sigma_{lm}+%
	\frac{\partial \sigma_{ij}}{\partial x_l}\frac{\partial \sigma_{lm}}{%
		\partial x_{l^{^{\prime }}}})(\rho_5)+(\frac{\partial^2 \sigma_{ij}}{%
		\partial x_l \partial x_{l^{^{\prime }}}}\sigma_{lm}+\frac{\partial
		\sigma_{ij}}{\partial x_l}\frac{\partial \sigma_{lm}}{\partial
		x_{l^{^{\prime }}}})(\rho_6)\Big)  \notag \\
	&\cdot (X_{l^{^{\prime }},n+1}(t_k^n)-X_{l^{^{\prime
			}},n+1}^a(t_k^n))(\Delta B_{j,2k}^{n+1}\Delta B_{m,2k+1}^{n+1}-\Delta
		B_{m,2k}^{n+1}\Delta B_{j,2k+1}^{n+1})
		\end{align}
		where all the $\rho_i$ and $\rho^{\prime}_i$ lie somewhere between $%
		X_{n+1}(t_k^n)$ and $X_{n+1}^a(t_k^n)$. For the sake of completeness, the other terms in \eqref{REINK} 
		from $M^{f,(1)}_{i,n,k}$ can be written as
		\begin{align}
		M^{f,(1)}_{i,n,k} \triangleq& \Big(\sum_{j=1}^d \frac{\partial \mu_i^{(n+1)}%
		}{\partial x_j}(X_{n+1}(t_{k}^n))\sum_{m=1}^{d^{\prime}}
		\sigma_{jm}(X_{n+1}(t_{k}^n))\Delta B_{m,2k}^{n+1} \frac{\Delta t_n}{2} \Big)
		\notag \\
		M^{f,(2)}_{i,n,k}\triangleq& \Big(\sum_{j=1}^{d^{\prime}}(%
		\sigma_{ij}(X_{n+1}(t^{n+1}_{2k+1}))-\sigma_{ij}(X_{n+1}(t^{n}_{k})))-%
		\sum_{j,m=1}^{d^{\prime}}\sum_{l=1}^d \big(\frac{\partial \sigma_{ij}}{\partial x_l}%
		\cdot\sigma_{lm}\big)(X_{n+1}(t_{k}^n))\Delta B_{m,2k}^{n+1}\Big)\Delta
		B_{j,2k+1}^{n+1}  \notag \\
		=&\sum_{j=1}^{d^{\prime}} \sum_{m=1}^d \frac{\partial \sigma_{ij}}{\partial x_m}%
		(X_{n+1}(t_k^n))\Big({\mu}_m^{(n+1)}(X_{n+1}(t_k^n))\frac{\Delta t_n}{2} 
		\notag \\
		&+\sum_{j^{^{\prime }},r=1}^{d^{\prime}}\sum_{l=1}^d\frac{\partial \sigma_{mj^{^{\prime
					}}}}{\partial x_{l^{^{\prime }}}}(X_{n+1}(t_{k}^n))%
					\sigma_{lr}(X_{n+1}(t_{k}^n))\widetilde{A}_{rj^{^{\prime
							}}}(t_{k}^n,t_{2k+1}^{n+1})\Big)\Delta B_{j,2k+1}^{n+1}  \notag \\
							&+\frac{1}{2} \sum_{j=1}^{d^{\prime}}\sum_{m,l=1}^d \frac{\partial ^2 \sigma_{ij}}{%
								\partial x_m \partial x_l}%
							(\rho_7)(X_{m,n+1}(t_{2k+1}^{n+1})-X_{m,n+1}(t_{k}^{n}))((X_{l,n+1}(t_{2k+1}^{n+1})-X_{l,n+1}(t_{k}^{n})))\Delta B_{j,2k+1}^{n+1}
							\notag \\
							M^{f,(3)}_{i,n,k}\triangleq& \sum_{j,m=1}^{d^{\prime}}\sum_{l=1}^d \Big((\frac{\partial
								\sigma_{ij}}{\partial x_l}\cdot\sigma_{lm})(X_{n+1}(t^{n+1}_{2k+1}))-(\frac{%
								\partial \sigma_{ij}}{\partial x_l}\cdot\sigma_{lm}(X_{n+1}(t_{k}^n)))\Big)%
							\widetilde{A}_{mj}(t_{2k+1}^{n+1},t_{k+1}^n)  \notag \\
							=&\sum_{j,m=1}^{d^{\prime}}\sum_{l,l^{\prime}=1}^d\Big(\frac{\partial ^2 \sigma_{ij}}{%
								\partial x_l \partial x_{l^{^{\prime }}}}\sigma_{lm}+\frac{\partial
								\sigma_{ij}}{\partial x_l}\frac{\partial \sigma_{lm}}{\partial
								x_{l^{^{\prime }}}}\Big)(\rho_8)(X_{l^{^{\prime
									}},n+1}(t_{2k+1}^{n+1})-X_{l^{^{\prime }},n+1}(t_{k}^{n}))\widetilde{A}%
								_{mj}(t_{2k+1}^{n+1},t_{k+1}^n)
								\end{align}
								where all the $\rho_i$ lie somewhere between $X_{n+1}(t^{n+1}_{2k+1})$ and $%
								X_{n+1}(t^{n}_{k}))$. Based on the expansion above, we use the fact that $%
								\mathbb{E}\big(B(t)-B(s)\big)^4=O(t-s)^2$ and $\mathbb{E}\big(\widetilde{A}%
								_{ij}(s,t)\big)^4=O((t-s)^4)$, independence of $X_{n+1}(t_k^n)$ (and $%
								X_{n+1}^a(t_k^n)$) with $\Delta B^n_k$, Lemma~\ref{note lemma3} and Lemma~%
								\ref{note lemma 5} to perform similar analysis on these terms like we did
								for $N^f_{i,n,k}$ and $N^{(1)}_{i,n,k}$. For convenience, we omit the
								details and conclude that we can find polynomial $\mathcal{P}(x) > 1$ for $x>1$ such
								that 
								\begin{align}  \label{remainorder}
								\mathbb{E} R^4_{i,n,k} \leq \mathcal{P}(L_1)\Delta
								t_n^{8(2\alpha-\beta)+2}.
								\end{align}
								Now we are ready to prove the hypothesis in \eqref{inductionhypothesis} by
								induction on $0 \leq k \leq 2^n$. First of all, when $k=0$, for $1 \leq i
								\leq d$, the claim holds since $\xi_{i,n,0}\triangleq
								X_{i,n}^{\mu}(0)-X_{i,n}(0) =x-x=0$. 
								
								Now, fixing $0 \leq k \leq 2^n-1$  and $1 \leq i \leq d$, 
								suppose the induction hypothesis holds so that we can find $C>1$ where 
								\begin{equation}  \label{inductionassumption2}
								\mathbb{E}\lvert \xi^4_{i,n,j}\rvert \leq
								e^{CL_1\cdot t_j^n}\cdot \Delta t_n^{4-\delta}
								\end{equation}
								for all $0 \leq j \leq k$. We want to show 
								\begin{equation}
								\mathbb{E}\lvert \xi^4_{i,n,k+1}\rvert \leq
								e^{CL_1\cdot t_{k+1}^n}\cdot \Delta t_n^{4-\delta}
								\end{equation}
								for all $1\leq i \leq d$. To achieve this, we again use \eqref{first big 2}
								\begin{align}  \label{first big 3}
								&\mathbb{E}(\xi_{i,n,k+1} ^4) \nonumber\\
								= & \mathbb{E}(\xi_{i,n,k} ^4)+\mathbb{E}(\eta_{i,n,k} ^4)+3\mathbb{E}_{\mu(\omega)}(\xi_{i,n,k}
								^3\eta_{i,n,k} )+3\mathbb{E}(\xi_{i,n,k} \eta_{i,n,k} ^3)+6%
								\mathbb{E}(\xi_{i,n,k} ^2\eta_{i,n,k} ^2)
								\end{align}
								to provide upper bounds for terms in \eqref{first big 3}.
								
								 We start with $%
								\eta^4_{i,n,k}$ by observing \eqref{etaagain} and using \eqref{remainingfield}
								to find constant $C>1$ that: 
								\begin{align}
								(\eta_{i,n,k})^4\leq& C (L_1^4\|\xi_{n,k}\|_{\infty}^4\Delta
								t_n^4+L^4\|\xi_{n,k}\|_{\infty}^4\Delta
								(B_{j,k}^n)^4+L^4\|\xi_{n,k}\|_{\infty}^4(\widetilde{A}%
								_{mj}(t_{k}^n,t_{k+1}^n))^4  \notag \\
								&+\lvert R_{i,n,k}\rvert^4+L_1^4\Delta t_n^{4(3+\frac{q-4}{2}%
									)\gamma+4})
								\end{align}
								where we can use the fact that $\mathbb{E}\big(B(t)-B(s)\big)^4=O(t-s)^2$
								and $\mathbb{E}\big(\widetilde{A}_{ij}(s,t)\big)^4=O((t-s)^4)$ and %
								\eqref{remainorder} to conclude: 
								\begin{align}  \label{etabound2}
								\mathbb{E}\eta_{i,n,k}^4 \leq& C (L_1^4\mathbb{E}%
								\|\xi_{n,k}\|_{\infty}^4\Delta t_n^4+L^4\mathbb{E}%
								\|\xi_{n,k}\|_{\infty}^4\Delta t_n^2+L^4\mathbb{E}%
								\|\xi_{n,k}\|_{\infty}^4\Delta t_n^4  \notag \\
								&+\mathcal{P}(L_1)\Delta t_n^{8(2\alpha-\beta)+2}
								+L_1^4(\omega)\Delta t_n^{4(3+\frac{q-4}{2})\gamma+4})  \notag \\
								\leq& e^{C_1L_1(\omega)t_k^n}\Delta t_n^{5-\delta} (\mathcal{P}%
								(L_1)\Delta
								t_n^{1+8(2\alpha-\beta)-(4-\delta)}+2L_1^4\Delta t_n^3 +2L^4\Delta
								t_n),
								\end{align}
								where the last line follows from the induction hypothesis, the fact that $(3+\frac{%
									q-4}{2})\gamma > 1$ , $4-\delta < 8(2\alpha-\beta)$ in Definition~\ref{alni} and the results in \eqref{remainorder}.
								
								For the bound on $\mathbb{E}(\xi_{i,n,k}^3\eta_{i,n,k})$, because
								of the independence of Brownian increments with $\mu(\cdot)\in\mathcal{L}_1$, we can simplfy
								the Equation \eqref{etaagain} using the martingale property and write 
								\begin{align}
								&\mathbb{E}(\xi_{i,n,k}^3\eta_{i,n,k})  \notag \\
								=&\mathbb{E}\bigg[(\xi_{i,n,k})^3\big((\mu^{(n)}_i(\bar{X}_{n+1}
								(t_{k}^n))-\mu^{(n)}_i(X_{n}(t_{k}^n)))\Delta
								t_n+N_{i,n,k}^{(1)}+N_{i,n,k}^{f}+({\mu}^{(n+1)}_i-{\mu}^{(n)}_i)(\bar{X}%
								_{i,n+1}(t_{k}^n))\Delta t_n)\big) \bigg] \notag \\
								\leq& L_1\bigg(\mathbb{E}\|\xi_{n,k}\|_{\infty}^4\Delta t_n
								+ (\mathbb{E} \|\xi_{n,k}\|_{\infty}^4)^{\frac{3}{4}}(\mathbb{E%
								}(N_{i,n,k}^{(1)})^4)^{\frac{1}{4}}+(\mathbb{E}
								\|\xi_{n,k}\|_{\infty}^4)^{\frac{3}{4}}(\mathbb{E}(N_{i,n,k}^{f})^4)^{\frac{1}{4}}  +(\mathbb{E}\|\xi_{n,k}\|_{\infty}^4)^{\frac{3}{4}}
								\Delta t_n^{(3+\frac{q-4}{2})\gamma+1}  \bigg)\notag \\
								\leq&e^{CL_1t_k^n}\Delta
								t_n^{5-\delta}(2L_1+2L_1\Delta t_n^{\frac{\delta}{4}%
									+2(2\alpha-\beta)-1})\notag
								\end{align}
								where the second inequality follows from H\"older's inequality and Equation \eqref{remainingfield}. The
								last inequality follows from the induction hypothesis, Equations \eqref{finalfinal1}, 
								\eqref{finalfinal2} and the fact that $(3+\frac{q-4}{2})\gamma > 1$ and $%
								8(2\alpha-\beta)>4-\delta$ as in Definition~\ref{alni}.
								
								Similarly, we have
								\begin{align}
								\mathbb{E}\xi_{i,n,k}(\eta_{i,n,k})^3 \leq& (\mathbb{E}%
								(\xi_{i,n,k} )^4)^{\frac{1}{4}} (\mathbb{E}(%
								\eta_{i,n,k} )^4)^{\frac{3}{4}}  \notag \\
								\leq& e^{CL_1t_k^n}\Delta t_n^{5-\delta} (\mathcal{P}%
								(L_1)+2L_1^4+2L^4)\Delta t_n^{\frac{1}{2}}  \notag \\
								\mathbb{E}(\xi_{i,n,k})^2(\eta_{i,n,k})^2 \leq &(\mathbb{E}%
								(\xi_{i,n,k} )^4)^{\frac{1}{2}}(\mathbb{E}(%
								\eta_{i,n,k})^4)^{\frac{1}{2}}  \notag \\
								\leq& e^{CL_1t_k^n}\Delta t_n^{5-\delta} \big(\mathcal{P}%
								(L_1)\Delta
								t_n^{8(2\alpha-\beta)-(4-\delta)}+2L_1^4\Delta t_n^2 +2L^4\big)^{%
									\frac{1}{2}}
								\end{align}
								following from H\"older's inequality and Equation \eqref{etabound2}.

								Let $C=5+2L>1$ and find
								polynomial $\mathcal{P}'(x) > 1$ for $x>1$ such that when $2^n>
									\mathcal{P}'(L_1)$, we have 
								\begin{align}  \label{whatwehave}
								(2L_1+2L_1\Delta t_n^{\frac{\delta}{4}+2(2\alpha-\beta)-1})
								\leq &3L_1  \notag \\
								(\mathcal{P}(L_1)\Delta
								t_n^{1+8(2\alpha-\beta)-(4-\delta)}+2L_1^4\Delta t_n^3 +2L^4\Delta
								t_n) \leq &1  \notag \\
								(\mathcal{P}(L_1)+2L_1^4+2L^4)\Delta t_n^{\frac{1}{2}}
								\leq& 1  \notag \\
								\big(\mathcal{P}(L_1)\Delta
								t_n^{8(2\alpha-\beta)-(4-\delta)}+2L_1^4\Delta t_n^2 +2L^4\big)^{%
									\frac{1}{2}} \leq &2L^2
								\end{align}
								Now we are ready to prove the induction hypothesis. In particular, when $2^n>
								\mathcal{P}'(L_1)$, we use the bound in Equations \eqref{first big 3}
								and \eqref{whatwehave} to obtain
								\begin{align}  \label{final1}
								\mathbb{E}(\xi_{i,n,k+1})^4 \leq &
								e^{CL_1t_k^n}\Delta t_n^{4-\delta}+ e^{CL_1t_k^n}\Delta
								t_n^{5-\delta}(3L_1+2+2L^2)  \notag \\
								\leq & e^{CL_1t_k^n}\Delta
								t_n^{4-\delta}\cdot(1+CL_1\Delta t_n)  \notag \\
								\leq & e^{CL_1t_{k+1}^n}\Delta t_n^{4-\delta}
								\end{align}
								where the last line follows from convexity of exponential function  $e^y
								\geq e^x + e^x\cdot(y-x)$ for $y\geq x$. Now we can use the same method as in the proof of
								Lemma~\ref{orderfirstlemma} to extend the induction hypothesis to the case
								where $\Delta t_n \leq \mathcal{P}'(L_1)$ and finish the proof of \eqref{mmmm1} and thus  the proof of Lemma~\ref{finitefourthmoment}.

\end{proof}

\subsection{Proof of Lemma~\ref{lemma L1}}\label{lookhere}
We first present a useful supporting lemma for the proof of Lemma~\ref{lemma L1}.

\begin{lemma}
	\label{mgt Gaussian} Fixing $\epsilon >0$, let $\{\boldsymbol{Z}_n\}_{n\geq 1}$ be a sequence I.I.D. standard $d$ dimensional Gaussian random vectors (i.e., $\Sigma_n=I_d$ for all $n\geq1$). Then, the random variable defined as
	\begin{equation}
	M_{\epsilon} \triangleq
	\sup_{n\geq 1} \frac{ \|\boldsymbol{Z}_n\|_{\infty}}{n^{\epsilon}},
	\end{equation} 
	 has finite moment-generating
	function (i.e., $\mathbb{E}[e^{tM_{\epsilon}}]<\infty$)  for all $t\geq 0$.
\end{lemma}

\begin{proof}[{Proof of Lemma~\ref {lemma L1}}]\label{proofL1}
	
	Let $\{\boldsymbol{V}_{n}\}_{n \geq 1}$ be a sequence of independent $d$ dimensional Gaussian random vectors with distribution  $\mathcal{N}(\boldsymbol{0},\Sigma_n)$, where the covariance matrix $\Sigma_n$ satisfies,
	\begin{equation}
	\|\Sigma_n\|_{F}<L,	\end{equation}
	for all $n\geq1$ as in Assumption \ref{assump:1}. Since $\Sigma_n$ are positive semi-definite matrices, each of them has a unique positive semi-definite square root matrix
	 $\Sigma_n^{\frac{1}{2}}$ \cite{Matrixana}. Moreover, we notice that 
	  \begin{align}
	  \|\Sigma_n^{\frac{1}{2}}\|^2_{F}=&{trace((\Sigma_n^{\frac{1}{2}})^T(\Sigma_n^{\frac{1}{2}}))} \nonumber\\
	  =& trace(\Sigma_n) \nonumber\\
	  =& \sum_{i} \lambda_{i}\nonumber\\
	  \leq & \sum_{i} (\lambda^2_{i}+1) \leq trace(\Sigma_n^T\Sigma_n)+d = \|\Sigma_n\|^2_{F} +d \leq L^2+d
	  \end{align}
	  where $\lambda_i$ are the eigenvalues of $\Sigma_n$. Thus, if we set $L'=\sqrt{L^2+d}>1$, we have $\|\Sigma^{\frac{1}{2}}_n\|_{F}<L' $
	   for all $n \geq 1$ and some $L'>1$. Finally, by the equivalence of matrix norms, we can further find some $L''>1$ such that $\|\Sigma_n^{\frac{1}{2}}\|_{\infty}<L''$,
		for all $n\geq1$ where $\|\cdot\|_{\infty}$ denotes the $L_{\infty}$ norms for matrices. 
		
		Consequently, if we let $\{\boldsymbol{Z}_n\}_{n\geq 1}$ be a sequence of I.I.D. $d$ dimensional  standard Gaussian random vectors and notice that $\{\Sigma^{\frac{1}{2}}_n\cdot \boldsymbol{Z}_n\}_{n \geq 1}$ follows the same  distribution as $\{\boldsymbol{V}_n\}_{n \geq 1}$, we can then define 
	\begin{equation}
	M_{\frac{q-4}{2}}\triangleq \sup\limits_{n\geq 1} \frac{\|\Sigma_n^{\frac{1}{2}}\cdot \boldsymbol{Z}_n\|_{\infty}}{n^{\frac{q-4}{2}}} \leq L'' \sup\limits_{n\geq 1} \frac{\| \boldsymbol{Z}_n\|_{\infty}}{n^{\frac{q-4}{2}}}.
	\end{equation} 
 It then simply follows from Lemma~\ref{mgt Gaussian} that, the random variable $M_{\frac{q-4}{2}}$ has finite moment-generating function for all $t\geq 0$. Thus, if we define 
 \begin{equation}
 N_{\frac{q-4}{2}} \triangleq \sup\limits_{n\geq 1} \frac{\| \boldsymbol{V}_n\|_{\infty}}{n^{\frac{q-4}{2}}},
 \end{equation}
 then the random variable $N_{\frac{q-4}{2}}$ would also have finite moment-generating function for all $t\geq 0$. Finally, according to Assumptions~\ref{assump:1} and~\ref{assump:2}, we have
\begin{equation}
\frac{\|\mu(x)-\mu(y)\|_{\infty}}{\|x-y\|_{\infty}}\leq\sum_{n=1}^{\infty}\frac{\abs{\lambda_n}}{n^{4+\frac{q-4 }{2}}} \frac{\|V_n\|_{\infty}}{n^{\frac{q-4}{2}}}{\frac{\|\psi_n(x)-\psi_n(y)\|_{\infty}}{\|x-y\|_{\infty}}}\leq \sum_{n=1}^{\infty}\frac{\abs{\lambda_n}}{n^{4+\frac{q-4}{2}}} \cdot N_{\frac{q-4}{2}} \cdot nL \leq CN_{\frac{q-4}{2}},
\end{equation}
for some constant $C>1$, which provides a bound for $\|\frac{\partial \boldsymbol{\mu}_i}{\partial x_l}\|_{\infty}$ with finite moment-generating function on the real line.
Using a similar method, we can find bounds with finite moment-generating function on the real line for  $\|\boldsymbol{\mu}\|_{\infty}$ and $\| \frac{\partial^2 \boldsymbol{\mu}_{i}}{\partial x_k\partial x_l}%
\|_{\infty} $ as well. The same bound applies for $\boldsymbol{S}_n$(thus $\boldsymbol{\mu}^{(n)}$, since $\boldsymbol{\mu}^{(n)}=\boldsymbol{S}_{\lfloor 2^{n\gamma} \rfloor}$) and $\bar{\boldsymbol{\mu}}^{(n)}$ for all $n$, and we can thus define a uniform bound finite moment-generating function for all these quantities on the real line denoted by $\boldsymbol{L}_1$. The requirement $\boldsymbol{L}_1>1$ can be added without affecting the result.
\end{proof}

\section{Proof of Supporting Lemmas}

First,  we introduce the following Levy-Ciesielski construction of the Brownian motion
(see, for example \cite{steele2012stochastic}) for the understanding of supporting Lemma~\ref{lcyes}.

\begin{lemma}
	\label{Uhere} Let $\{U_j^m : 1\leq j \leq 2^{m-1},m
	\geq 1\}$ along with $U_0^0$ be a sequence of I.I.D standard normal
	random variables, and we define 
	\begin{equation}
	H(t)\triangleq \mathbf{I}(0\leq t < 1/2)-\mathbf{I}(1/2 \leq t \leq 1)
	\end{equation}
	along with its family of functions $\{H_j^m(t)=2^{m/2}H(2^{m-1}t-j+1): 1\leq j \leq 2^{m-1},m
	\geq 1\}$ and constant function $H_0^0(\cdot)=1$. Now, if we define $B(t)$ for $t \in [0,1]$ by 
	\begin{equation}  \label{levyco}
	B(t)\triangleq U_0^0 \int_0^t H_0^0(s)ds + \sum_{m \geq
		1}\sum_{j=1}^{2^{m-1}}\Big( U_j^m\int_0^t H_j^m(s)ds\Big),
	\end{equation}
	then it can be shown that the right-hand side converges uniformly on [0,1]
	almost surely and the process $\{ B(t) : t \in [0,1]\}$ is a standard
	Brownian motion on [0,1].
\end{lemma}
\begin{proof}
	See Section 2.3 of \cite{karatzas2012brownian}.
\end{proof}

This
theoretical construction provides a way to sample Brownian motion by
sampling independent Gaussian random variables. Here, for ${d^{\prime}}$-dimensional Brownian
motion we use ${d^{\prime}}$-dimensional Gaussian random variables. Furthermore, using Lemma~\ref{Uhere} and the fact that changing the sign of a standard Gaussian variable does not change its distribution, we have the following corollary on $B^{(n+1),a}(t)$ related to Definition~\ref{adefn}.
\begin{coro}
	Fixing $n \geq 0$ and the sequence of I.I.D. standard normal
	random variables $\{U_j^m : 1\leq j \leq 2^{m-1},m
	\geq 1\}$ along with $U_0^0$, we can define
	\begin{equation}  \label{levycoa}
	B^{n+1,a}(t)\triangleq U_0^0 \int_0^t H_0^0(s)ds +\sum_{j=1}^{2^{n}}\Big(- U_j^{n+1}\int_0^t H_j^{n+1}(s)ds\Big)+ \sum_{\substack{ m \geq 1 
			\\ m \neq n+1}}\sum_{j=1}^{2^{m-1}}\Big( U_j^m\int_0^t H_j^m(s)ds\Big),
	\end{equation}
	which is a Brownian motion on [0,1].
\end{coro}

\begin{lemma}
	\label{lcyes1} Given a sequence of I.I.D. standard normal $\{U_j^m : 1\leq j \leq 2^{m-1},m
	\geq 1\}$ along with $U_0^0$ and fixing $n \geq 0$, define $B(t), 0\leq t \leq 1$ as in \eqref{levyco} and $B^{n+1,a}(t),0\leq t \leq 1$ as in %
	\eqref{levycoa}. Then for $1\leq k \leq 2^{n+1}$, let 
	\begin{align}
	\Delta B_k^{n+1} =&
	B(t^{n+1}_{k+1})-B(t_k^{n+1})\nonumber\\
	\Delta B_k^{n+1,a}=&B^{(n+1),a}(t^{n+1}_{k+1})-B^{(n+1),a}(t^{n+1}_k).
	\end{align}
	Then $\Delta B_k^{n+1}$ and $\Delta B_k^{n+1,a}$ satisfy equations  \eqref{antitheticB} and thus \eqref{trihere} in Definition~\ref{adefn}. Thus, we may regard $X_{n+1}^a(\cdot)$ to be an
	antithetic scheme $X_{n+1}(\cdot)$ generated under Brownian motion $B^{n+1,a}(\cdot)$ instead of $B(\cdot)$.
\end{lemma}

\begin{proof}[Proof of Lemma~\ref{lcyes1}]
	Following Definition~\ref{Uhere}, fixing  $n \geq 1$ and $0 \leq k \leq 2^{n-1}-1$, we observe that  
	\begin{equation}\label{Bproof0}
	\begin{cases}
	\int_{t^{n}_{2k}}^{t^{n}_{2k+1}}H^m_j(t) dt =  \int_{t^{n}_{2k+1}}^{t^{n}_{2k+2}}H^m_j(t) dt &\quad \text{for all} \quad m \neq n \quad \text{and} \quad 1\leq j \leq 2^{m-1} \\
	\int_{t^{n}_{2k}}^{t^{n}_{2k+1}}H^m_j(t) dt =  -\int_{t^{n}_{2k+1}}^{t^{n}_{2k+2}}H^m_j(t) dt &\quad \text{for all} \quad m = n \quad \text{and} \quad 1\leq j \leq 2^{m-1}.
	\end{cases}
	\end{equation}
 Thus, we have that, for $0\leq k \leq 2^{n}-1$,
	\begin{align}
	B^{n+1,a}(t^{n+1}_{2k+1})-B^{n+1,a}(t^{n+1}_{2k})=&B(t^{n+1}_{2k+2})-B(t^{n+1}_{2k+1})= \Delta B_{2k}^{n+1,a}\notag\\
	B^{n+1,a}(t^{n+1}_{2k+2})-B^{n+1,a}(t^{n+1}_{2k+1})=&B(t^{n+1}_{2k+1})-B(t^{n+1}_{2k})= \Delta B_{2k+1}^{n+1,a}
	\end{align}
	by simply taking the difference in \eqref{levycoa} and checking \eqref{Bproof0}.
\end{proof}

\begin{proof}[Proof of Lemma \protect\ref{jose1}]
	Following Definition~\ref{add15}, define $R_{i,j}^n(t_l^n,t_m^n) = \sum_{k=l+1}^m {A}_{i,j}(t_{k-1}^n,t_k^n)$ for $ 0\leq l < m \leq 2^n, 1 \leq i,j \leq d^{\prime}$ and $i\neq j$. Then, we can define  
	\begin{equation}
	\Gamma _{{R}%
	}\triangleq \sup_{n \geq 1}\sup_{\substack{ 0\leq s\leq t\leq 1  \\ s,t\in D_{n}}}%
	\max_{1\leq i,j\leq d^{\prime}, i \neq j}\frac{\lvert {R}_{i,j}^{n}(s,t)\rvert }{%
		\lvert t-s\rvert ^{\beta }\Delta t_{n}^{2\alpha -\beta }} \quad \text{and} \quad \Gamma _{R-\tilde{R}%
	}\triangleq \sup_{n \geq 1}\sup_{\substack{ 0\leq s\leq t\leq 1  \\ s,t\in D_{n}}}%
	\max_{1\leq i,j\leq d^{\prime}, i \neq j}\frac{|R^n_{i,j}(s,t)-\tilde{R}^n_{i,j}(s,t)|}{%
		\lvert t-s\rvert ^{\beta }\Delta t_{n}^{2\alpha -\beta }},
	\end{equation}
	Observing the definition for both the case $i=j$ and $i\neq j$, we have the following bound: \begin{equation}\label{add12}
	\|\tilde{A}\|_{2\alpha} \leq \|A\|_{2\alpha}+\|B\|^2_{\alpha}\qquad \text{and} \qquad\Gamma_{\tilde{R}} \leq \Gamma_{R}+\Gamma_{R-\tilde{R}}.
	\end{equation}
	Now, following Lemma 3.1 in \cite{blanchet2017}, we define a family of random variables ($L^n_{i,j}(k): k=0,1,...,2^{n-1}, 1\leq i,j \leq d^{\prime} , i\neq j, n \geq 1$) satisfying:
	\begin{align}
	L_{i,j}^{n}\left( 0\right) &=0\nonumber\\
	L_{i,j}^{n}\left( k\right) &=L_{i,j}^{n}\left( k-1\right) +\left( B_{i}\left(
	t_{2k-1}^{n}\right) -B_{i}\left( t_{2k-2}^{n}\right) \right) \left(
	B_{j}\left( t_{2k}^{n}\right) -B_{j}\left( t_{2k-1}^{n}\right) \right).
	\end{align} 
	Then, following Lemma 3.4 and its proof in \cite{blanchet2017}, we define, for $1 \leq i,j \leq d^{\prime}$ and $i \neq j$,
	\begin{equation}
	N_{i,j,2}=\max\{n:|L^n_{i,j}(m)-L^n_{i,j}(l)|>(m-l)^{\beta}\Delta t^{2\alpha}_n \quad\text{for some} \quad 0 \leq l <m \leq 2^{n-1}\}, 
	\end{equation} 
	and define $N_2=\max\{ N_{i,j,2}:1\leq i,j \leq d^{\prime}, i\neq j \}$ along with
	\begin{equation*}
	\Gamma _{L}\triangleq\max \{1,\max_{1\leq i,j\leq d^{\prime},i\neq j}\max_{n<N_{2}}\max_{0\leq
		l<m<2^{n-1}}\frac{|L_{i,j}^{n}\left( m\right) -L_{i,j}^{n}\left( l\right) |}{%
		\left( m-l\right) ^{\beta }\Delta t_n^{2\alpha }}\}.
	\end{equation*}%
	Finally, we can use Definition~\ref{add15} and apply the result of Lemma 3.5 in \cite{blanchet2017} to write:
	\begin{align}\label{add13}
	\Gamma_{R} \leq \frac{2^{-(2\alpha-\beta)}}{1-2^{-(2\alpha-\beta)}}\cdot\Gamma_{L}, \qquad \text{and} \qquad \|A\|_{2\alpha} \leq \Gamma_{R}\cdot\frac{2}{1-2^{-2\alpha}}+\|B\|^2_{\alpha}\cdot\frac{2^{1-\alpha}}{1-2^{-\alpha}}.
	\end{align}
	
	Combining the result from \eqref{add12} and \eqref{add13},  to conclude the proof, it suffices to show that $\|B\|_{\alpha},\Gamma_{L}$ and $\Gamma_{R-\tilde{R}}$ has finite moments of every order. The fact
	that $\Vert B\Vert _{\alpha }$ has finite moments of every order follows
	from Borell's inequality for continuous Gaussian random fields (see Section 2.3 of \cite{adler2003random}). To show that $\Gamma_{L}$ has finite moments of every order, we first follow the proof of Lemma 3.4 in \cite{blanchet2017} to show that 
	\begin{align}
	\mathbb{P}(N_{i,j,2}\geq n) &\leq \sum_{h=n}^{\infty}\mathbb{P}(|L_{i,j}^h(m)-L_{i,j}^h(l)| > (m-l)^{\beta} \Delta t_n^{2\alpha} \quad \text{for some} \quad 0\leq l< m \leq 2^{n-1}) \nonumber\\
	&\leq  \sum_{h=n}^{\infty}2^{2h} \exp(-\theta^{\prime}2^{h(1-2\alpha)}) \nonumber\\
	&\leq \exp(-\frac{\theta^{\prime}}{2}\cdot 2^{n(1-2\alpha)}) \sum_{h=0}^{\infty}2^{2h} \exp(-\frac{\theta^{\prime}}{2}\cdot 2^{h(1-2\alpha)}) \nonumber \\
	& \leq C\exp(-\frac{\theta^{\prime}}{2}\cdot 2^{n(1-2\alpha)})
	\end{align}
	for some $C>1$ and $\theta^{\prime}>0$. It follows that,
	\begin{equation}\label{Tail_N2}
	\mathbb{P}( N_{2}\geq n) \leq C(d^{\prime})^2 \exp(-\frac{\theta^{\prime}}{2}\cdot 2^{n(1-2\alpha)}) . 
	\end{equation}
	Therefore, we can show
	
	\begin{equation}\label{add9}
	\mathbb{E}\left( \exp \left( \eta N_{2}\right) \right) \leq \sum_{n=1}^{\infty }C(d^{\prime})^2\exp
	\left( \eta n\right)  \exp(-\frac{\theta^{\prime}}{2}\cdot 2^{n(1-2\alpha)}) <\infty 
	\end{equation}%
	for every $\eta>0$. On the other hand, since for $m>l, n\leq N_2$, we have
	\begin{equation}
{(m-l)^{-\beta}\Delta t_n^{-2\alpha}} ={(m-l)^{-\beta}2^{2\alpha n}}  \leq 2^{2\alpha N_2},
	\end{equation}
	\begin{equation} 
	\Gamma_{L} \leq 1+2^{2\alpha N_2}\cdot\big(\max_{1\leq i,j\leq d^{\prime},i\neq j}\max_{n<N_{2}}\max_{0\leq
		l<m<2^{n-1}}{|L_{i,j}^{n}\left( m\right) -L_{i,j}^{n}\left( l\right) |}\big).
	\end{equation} 
	Since $N_{2}$ has a finite moment-generating
	function on the whole real line according to \eqref{add9},  in order to establish that $\Gamma _{L}$ has
	finite moments of every order, it suffices to show that 
	\begin{equation*}
	\mathbb{E}\left[ \left( \sum_{n=1}^{N_{2}}\sum_{1<l<m<2^{n-1}}\sum_{1\leq i,j\leq
		d^{\prime},i\neq j}|L_{i,j}^{n}\left( m\right) -L_{i,j}^{n}\left( l\right) |\right) ^{k}%
	\right] <\infty 
	\end{equation*}%
	for every $k\geq 1$. Letting $\bar{n}$ be the number of total elements being
	summed up inside the previous expectation, it follows that  $\bar{n}\leq N_2\cdot 2^{2N_{2}}(d^{\prime})^{2}$
	and therefore, by \eqref{elementary inequality}, that
	\begin{align}\label{add10}
	&\mathbb{E}\left[ \left( \sum_{n=1}^{N_{2}}\sum_{1<l<m<2^{n-1}}\sum_{1\leq i,j\leq
		d^{\prime},i\neq j}|L_{i,j}^{n}\left( m\right) -L_{i,j}^{n}\left( l\right) |\right) ^{k}%
	\right]   \nonumber \\
	\leq &\mathbb{E}\left[ \bar{n}^{k-1}\sum_{n=1}^{N_{2}}\sum_{1<l<m<2^{n-1}}\sum_{1%
		\leq i,j\leq d^{\prime},i\neq j}|L_{i,j}^{n}\left( m\right) -L_{i,j}^{n}\left( l\right) |^{k}%
	\right]   \nonumber \\
	\leq & \sum_{n=1}^{\infty }\sum_{1<l<m<2^{n-1}}\sum_{1\leq i,j\leq d^{\prime},i\neq j}\mathbb{E}\left[
	(N_2\cdot 2^{2N_{2}}(d^{\prime})^{2})^{k-1}|L_{i,j}^{n}\left(
	m\right) -L_{i,j}^{n}\left( l\right) |^{k}I\left( N_{2}\geq n\right) \right].
	\end{align}
	To bound the term in \eqref{add10}, we first show that, fixing any $h\geq 1$, $\mathbb{E}|L_{i,j}^n(m)-L_{i,j}^n(l)|^h$ is uniformly bounded for any $n \geq 1, 1 \leq l<m \leq 2^{n-1},  1 \leq i,j \leq n$ and $i \neq j$. 
	
	Let $\{Y_{i^{\prime}}\}_{i^{\prime}\geq 1}$ be I.I.D. random variables such that $Y\stackrel{\mathcal{D}}{=}Z_1\cdot Z_2$ where $Z_1,Z_2$ are independent standard normal random variables. It follows from H\"older's inequality and Jensen's inequality that we can find $C_h > 0$ such that $\mathbb{E}|\frac{\sum_{i=1}^n Y_{i^{\prime}}}{n}|^h < C_h$ for all $n\geq 1$. Then $\mathbb{E}|L_{i,j}^n(m)-L_{i,j}^n(l)|^h < C_h$ follows from $|L_{i,j}^n(m)-L_{i,j}^n(l)|\stackrel{d}{=}|\Delta t_n\sum_{i^{\prime}=1}^{m-l} Y_{i^{\prime}} |\leq | \frac{\sum_{i^{\prime}=1}^{m-l} Y_{i^{\prime}}}{m-l}|$. Specifically $\mathbb{E}|L_{i,j}^n(m)-L_{i,j}^n(l)|^{4k} < C_{4k}$ for all $n \geq 1$. Now we can use H\"older's inequality multiple times and the fact that $N_2$ has moment-generating function to conclude:  
	\begin{align}
	E\left[ (d^{\prime})^{2\left( k-1\right) }2^{3N_{2}\left( k-1\right)
	}|L_{i,j}^{n}\left( m\right) -L_{i,j}^{n}\left( l\right) |^{k}I\left(
	N_{2}\geq n\right) \right]  &\leq C^{\prime }f\left( N_{2}\geq n\right) ^{1/2}
	\end{align}%
	for some $C^{\prime }>0$ and therefore, it follows from \eqref{Tail_N2} and \eqref{add10} that $\Gamma _{L}$ has moments of every order.
	
	Finally, to show that $\Gamma_{R-\tilde{R}}$ has finite moments of every order, we define another family of random variables ($\tilde{L}^n_{i,j}(k): k=0,1,...,2^{n}, 1\leq i,j \leq d^{\prime} , i\neq j, n \geq 1$) satisfying:
	\begin{align}
	\tilde{L}_{i,j}^{n}\left( 0\right) &=0\nonumber\\
	\tilde{L}_{i,j}^{n}\left( k\right) &=\tilde{L}_{i,j}^{n}\left( k-1\right) +\left( B_{i}\left(
	t_{k}^{n}\right) -B_{i}\left( t_{k-1}^{n}\right) \right) \left(
	B_{j}\left( t_{k}^{n}\right) -B_{j}\left( t_{k-1}^{n}\right) \right),
	\end{align} 
	and similarly define 
	\begin{equation}
	\tilde{N}_2=\max\{n:|\tilde{L}^n_{i,j}(m)-\tilde{L}^n_{i,j}(l)|>(m-l)^{\beta}\Delta t_n^{2\alpha} \quad \text{for some}\quad 0 \leq l<m \leq 2^n, 1\leq i,j\leq d^{\prime}, i\neq j\},
	\end{equation}
	\begin{equation*}
	\Gamma _{\tilde{L}}\triangleq\max \{1,\max_{1\leq i,j\leq d^{\prime},i\neq j}\max_{n<\tilde{N}_{2}}\max_{0\leq
		l<m<2^{n-1}}\frac{|\tilde{L}_{i,j}^{n}\left( m\right) -\tilde{L}_{i,j}^{n}\left( l\right) |}{%
		\left( m-l\right) ^{\beta }\Delta t_n^{2\alpha }}\}.
	\end{equation*}%
	Then, for $1 \leq i,j \leq d^{\prime}, i\neq j, n\geq 1$ and $0 \leq s < t \leq 1$, $s,t \in D_n$, we have
	\begin{equation} R_{i,j}^n(s,t)-\tilde{R}_{i,j}^n(s,t)=\sum_{k=s2^n+1}^{t2^n} \widetilde{A}_{i,j}(t_{k-1}^n,t_k^n)=\tilde{L}_{i,j}^n(t2^n)-\tilde{L}_{i,j}^n(s2^n),
	\end{equation}
	which implies $\Gamma_{R-\tilde{R}}\leq\Gamma_{\tilde{L}}$. We can now proceed to show $\Gamma_{\tilde{L}}$ has finite moments of every order in the similar fashion as we did for $\Gamma_{L}$. This completes the proof.
\end{proof}

\begin{proof}[Proof of Lemma~\ref{note lemma3}]
	Let $X^{\mathcal{M}}_{n}(\cdot)$ be the following Milstein discretization scheme with step size $2^{-n}$  :
	\begin{align}
	X^{\mathcal{M}}_{i,n}(t_{k+1}^n)&=X^{\mathcal{M}}_{i,n}(t_{k}^n)+\mu_i(X^{\mathcal{M}}_n(t_{k}^n))\Delta t_n +\sum_{j=1}^{d^{\prime}} \sigma_{ij}(X^{\mathcal{M}}_n(t_{k}^n))\Delta B_{j,k}^n\nonumber\\
	&+ \sum_{j=1}^{d^{\prime}}\sum_{l=1}^d\sum_{m=1}^{d^{\prime}} \frac{\partial \sigma_{ij}}{\partial x_l}(X^{\mathcal{M}}_{n}(t_{k}^n))\sigma_{lm}(X^{\mathcal{M}}_{n}(t_{k}^n))A_{mj}(t_{k}^n,t_{k+1}^n),
	\end{align}
	where we use $A_{ij}(s,t)$ instead of $\widetilde{A}_{ij}(s,t)$ defined in \eqref{jose1}. (This distinguishes $X^{\mathcal{M}}_{n}(\cdot)$ from $X_n(\cdot)$, our antithetic scheme.).
	Then, fixing $\mu\in\mathcal{L}_1$ and $\mu^{(n)}\in\mathcal{L}_1$ with bounding number $L_1$, we can compute constant $C_1$ explicitly in terms of $L_1,\|B\|_{\alpha}$ and $\|A\|_{2\alpha}$ (originally denoted as $M,\|Z\|_{\alpha}$ and $\Aaa$ in \cite{blanchet2017}) such that for $n$ large enough and $r,t \in D_n$,
	\begin{equation}
	\|X^{\mathcal{M}}_n(t)-X^{\mathcal{M}}_n(r)\|_{\infty} \leq C_1\abs{t-r}^{\alpha}
	\end{equation}
	See page $305$ of \cite[Lemma 6.1]{blanchet2017}. 
	To get the result for $X_n(\cdot)$ instead of $X^{\mathcal{M}}_n(\cdot)$, we follow page $283$ of \cite[Lemma 2.1]{blanchet2017}, replacing $\Aaa $ by $\Aaaa$ in notation, we define
	\begin{equation}
	\begin{cases}
	C_1(\delta)&=\bar{d} L_1\|B\|_{\alpha} +1/2 \\ C_2(\delta)&=\bar{d}^3 L_1^2 \|A\|_{2\alpha} +1/2\\   C_3(\delta)&=\frac{2}{1-2^{1-3\alpha}}(\bar{d} L_1 C_1(\delta)^2 \|B\|_{\alpha} +\bar{d}^2 L_1 C_2(\delta) \|B\|_{\alpha} +\bar{d} ^2 L_1^2 \|B\|_{\alpha}+2\bar{d}^3 L_1^2 C_1(\delta)\|A\|_{2\alpha})\\
	\end{cases}
	\end{equation} 
	and we then find some fixed polynomial  $\mathcal{P}(x)>1$ for $x>1$ so that if $\delta = (\mathcal{P}(L_1,\|B\|_{\alpha},\|A\|_{2\alpha}))^{-1}$, then
	\begin{equation}\label{deltasmall1}
	C_3(\delta) \delta^{2\alpha}+L_1\delta^{1-\alpha}+\dbar^3L_1^2\|A\|_{2\alpha} \delta^{\alpha} < 1/2 \qquad \text{also} \qquad
	C_3(\delta) \delta^{\alpha} < 1/2,
	\end{equation}
	so that Equation $(6.4)$ in page $308$ of \cite[Lemma 6.1]{blanchet2017} is satisfied: 
	\begin{equation}
	\begin{cases}
	C_1(\delta)&\geq \bar{d} L_1 \|B\|_{\alpha} +L_1\delta^{1-\alpha}+\bar{d}L_1\|B\|_{\alpha}+\bar{d}^3L_1^2\|\widetilde{A}\|_{2\alpha}\delta^{\alpha} \\ C_2(\delta)&\geq \bar{d}^3 L_1^2 \|A\|_{2\alpha} +\bar{d}^3L_1^2\|\widetilde{A}\|_{2\alpha} \\   C_3(\delta)&\geq \frac{2}{1-2^{1-3\alpha}}(\dbar L_1 C_1(\delta)^2 \|B\|_{\alpha} +\bar{d}^2 L_1 C_2(\delta) \|B\|_{\alpha} +\bar{d} ^2 L_1^2 \|B\|_{\alpha}+2\bar{d}^3 L_1^2  C_1(\delta)\|A\|_{2\alpha})\\
	\end{cases}
	\end{equation} 
	which gives, according to line $12-17$ of page $308$ of \cite[Lemma 6.1]{blanchet2017}, that 
	\begin{equation}\label{haha1}
	\|X_n(t)-X_n(r)\|_{\infty} \leq \frac{2}{\delta}C_1(\delta)\abs{t-r}^{\alpha}
	\end{equation}
	for all $n$ large enough where $\Delta t_n \leq \frac{1}{2} \delta$. Notice here we have changed the result to address $X_n(\cdot)$ instead of $X^{\mathcal{M}}_n(\cdot)$, and so far it just follows from an easy modification of \cite[Lemma 6.1]{blanchet2017}.\\
	Now, to extend the result for $n$ where $\Delta t_n > \frac{\delta}{2}$, notice the recursion step in \eqref{coarse} is carried out at most $2^n$ number of times where $2^n= (\Delta t_n)^{-1} < 2(\delta)^{-1}=2\mathcal{P}(L_1,\|B\|_{\alpha},\|A\|_{2\alpha})$. By analyzing $\eqref{coarse}$ term by term, we have
	\begin{align} \|X_n(t_{k+1}^n)-X_n(t_{k}^n)\|_{\infty} \leq& \bar{d}(C L_1\Delta t_n +\bar{d} L\|B\|_{\alpha}\Delta t_n^{\alpha}+\bar{d}^3L^2\|A\|_{2\alpha}\Delta t_n^{2\alpha}) \nonumber\\
	\leq& \bar{d}(C L_1 +\bar{d} L\|B\|_{\alpha}+\bar{d}^3L^2\|A\|_{2\alpha})\Delta t_n^{\alpha} 
	\end{align}
 for some $C>1$. Since $\Delta t_n < 1$, thus, for $\Delta t_n > \frac{\delta}{2}$,
	\begin{align}\label{haha2}
	\|X_n(t)-X_n(r)\|_{\infty} & \leq \frac{\abs{t-r}}{\Delta t_n}\bar{d}(C L_1 +\bar{d} L\|B\|_{\alpha}+\bar{d}^3L^2\|A\|_{2\alpha})\Delta t_n^{\alpha} \nonumber \\
	&\leq  \bar{d}(C L_1 +\bar{d} L\|B\|_{\alpha}+\bar{d}^3L^2\|A\|_{2\alpha})\abs{t-r}^{\alpha}\frac{1}{\Delta t_n^{1-\alpha}} \nonumber \\
	&\leq \bar{d}(C L_1 +\bar{d} L\|B\|_{\alpha}+\bar{d}^3L^2\|A\|_{2\alpha})\abs{t-r}^{\alpha}\frac{2^{1-\alpha}}{\delta^{1-\alpha}}  \nonumber \\
	&\leq 2\bar{d}(C L_1 +\bar{d} L\|B\|_{\alpha}+\bar{d}^3L^2\|A\|_{2\alpha})\cdot \mathcal{P}(L_1,\|B\|_{\alpha},\|A\|_{2\alpha}) \cdot \abs{t-r}^{\alpha} 
	\end{align}
	where the last line follows from $\mathcal{P}(x) > 1$ for $x>1$. The second to last line follows from $\Delta t_n > \frac{\delta}{2}$. We now  combine \eqref{haha1} and \eqref{haha2} and let
	\begin{align} \mathcal{P}'(L_1,\|B\|_{\alpha},\|A\|_{2\alpha}) \triangleq 2\bar{d}(C L_1 +\bar{d} L\|B\|_{\alpha}+\bar{d}^3L^2\|A\|_{2\alpha})\cdot \mathcal{P}(L_1,\|B\|_{\alpha},\|A\|_{2\alpha})\cdot \frac{2}{\delta}C_1(\delta)
	\end{align}
	be the polynomial where 
	\begin{equation}
	\|X_n(t)-X_n(r)\|_{\infty}\leq \mathcal{P}'(L_1,\|B\|_{\alpha},\|\widetilde{A}\|_{2\alpha}) \abs{t-r}^{\alpha}
	\end{equation}
	for all $n$. This completes the proof.
\end{proof}
\begin{proof}[Proof of Lemma~\ref{note lemma4}]
	The discretization $\hat{X}^n(\cdot)$ from Equation $(2.4)$ on page $280$ of \cite{blanchet2017}  is defiend as:
	\begin{align}\label{epdiscrete}
	\hat{X}_i^{n}(t_{k+1}^n)  
	=&\hat{X}_i^{n}(t_{k}^n)+\mu_i(\hat{X}^n(t_{k}^n))\Delta t_n +\sum_{j=1}^{d^{\prime}} \sigma_{ij}(\hat{X}^n(t_{k}^n))\Delta B_{j,k}^n\nonumber\\
	&+ \sum_{j=1}^{d^{\prime}}\sum_{l=1}^d\sum_{m=1}^{d^{\prime}} \frac{\partial \sigma_{ij}}{\partial x_l}(\hat{X}^{n}(t_{k}^n))\sigma_{lm}(\hat{X}^{n}(t_{k}^n))\hat{A}_{mj}(t_{k}^n,t_{k+1}^n)
	\end{align}
	where  $\hat{A}_{i,j}(s,t)=0$ for $ i\neq j$ and $\hat{A}_{i,i}(s,t)={A}_{i,i}(s,t)$ $ \forall 1\leq i \leq d$ as in Definition \ref{jose0}. Consequently, it is defined on   page $280$ of \cite{blanchet2017}, as in Definition~\ref{jose0}, that
	\begin{equation}
	{R}^n_{i,j}(t_l^n,t_m^n)\triangleq \sum_{k=l+1}^m \{A_{i,j}(t_{k-1}^n,t_k^n)-\hat{A}_{i,j}(t_{k-1}^n,t_k^n)\}  \text{ and }  \Gamma_{R} \triangleq \sup_{n}\sup_{\substack{0\leq s\leq t \leq 1 \\ s,t \in D_n}} \max_{1\leq i,j\leq {d^{\prime}}} \frac{\abs{{R}_{i,j}^n(s,t)}}{\abs{t-s}^{\beta}\Delta t_n^{2\alpha-\beta}}.
	\end{equation}
	
	With a slight change in notation, we replace $M$ with $L_1(\omega)$, $\|Z\|_{\alpha}$ with $\|B\|_{\alpha}$,  then according to \cite[Theorem ~2.1]{blanchet2017}, we can find constant $G$ (for notation consistency with \cite{blanchet2017}) explicitly in terms of $L_1,K_{\alpha},K_{2\alpha}$ and $K_{R}$ such that 
	\begin{equation}
	\|\hat{X}^n(t)-{X}_t\|_{\infty} \leq G\Delta t_n^{2\alpha-\beta}
	\end{equation}
	where we may take $K_{\alpha}=\|B\|_{\alpha},K_{2\alpha}=\|A\|_{2\alpha}$ and $K_{R}=\Gamma_{R}+1$. 
	
	To prove a similar result for $\|{X}_n^{\mu}(t)-{X}_t\|_{\infty}$ instead of  $\|\hat{X}^n(t)-{X}_t\|_{\infty}$, we  replace $\Gamma_{R}$ with our $\Gamma_{\widetilde{R}}$ defined in Definition \ref{jose0}, the proof will follow exactly as in the proof of Theorem $2.1$ in \cite{blanchet2017}[Proposition 6.1 and 6.2]. Particularly, we are able to compute constant $G$ in terms of $L_1(\omega), \|B\|_{\alpha},\|A\|_{2\alpha}$ and $\Gamma_{\widetilde{R}}$ such that 
	\begin{equation}
	\|X_n^{\mu}(t)-X_t\|_{\infty} \leq G\Delta t_n^{2\alpha-\beta}.
	\end{equation}
	Moreover, following Section $2.2$ on pages $282-283$ of \cite{blanchet2017} (part of which is shown in Lemma \eqref{note lemma3}), the construction of the constant $G$ only involves multiplication and addition among the variables $L_1$,$\|B\|_{\alpha}$,$\|A\|_{2\alpha}$, $\Gamma_{\widetilde{R}}$ and constants. This suggests that we can find some fixed polynomial  $\mathcal{P}''(\cdot)$ such that $\mathcal{P}''(x) > 1$ for $x>1$ and 
	\begin{equation}
	\|X_n^{\mu}(t)-X_t\|_{\infty} \leq \mathcal{P}''(L_1,\|B\|_{\alpha},\|A\|_{2\alpha},\Gamma_{\widetilde{R}})\Delta t_n^{2\alpha-\beta}.
	\end{equation}
\end{proof}

\begin{remark}
	A technical detail here is that the construction of the constant $G$ in
	Section $2.2$ of \cite{blanchet2017} for Theorem $2.1$ actually only makes
	the statement of Theorem $2.1$ valid for $n$ \textquotedblleft large"
	enough (see the proof of Theorem $2.1$ in Section $6$ of \cite{blanchet2017}
	[Proposition~6.1 and 6.2]). However, we may extend the result to hold
	for all $n$ using the similar method in our \eqref{haha2} of Lemma~\ref{note
		lemma3}. There we modified the proof to extend the result originally only
	valid for $n$ \textquotedblleft large" enough, meaning $\Delta t_{n}\leq 
	\frac{\delta }{2}$ for $\delta =(\mathcal{P}_{1}(L_{1},\Vert
		B\Vert _{\alpha },\Vert \widetilde{A}\Vert _{2\alpha }))^{-1}$, to all $n$ while
	still maintaining the bound  $\mathcal{P}'(\cdot )$ to be some polynomial
	of $L_{1}$,$\Vert B\Vert _{\alpha }$ and $\Vert \tilde{A}\Vert
	_{2\alpha }$. The situation is similar here, and thus by a similar but more lengthy argument, we can extend the result of Lemma \eqref{note
		lemma4} to hold for all $n$ while still making the upper bound of $G$ above,
	namely $\mathcal{P}''(\cdot )$, to be a polynomial of $L_{1}$,$%
	\Vert B\Vert _{\alpha }$,$\Vert A\Vert _{2\alpha }$.
\end{remark}

\begin{proof}[Proof of Lemma~\ref{lcyes}]
	Denote  $X(t;\mu,B), 0\leq t \leq 1$ to be the solution of SDE under field $\mu(\cdot)\in\mathcal{L}_1$ and Brownian motion $B(t), 0 \leq t \leq 1$. Let $X_n(t; \mu^{(n)},B)$ be our antithetic scheme under the field $\mu^{(n+1)}\in\mathcal{L}_1$ instead of $\mu^{(n)}$.  Since the Brownian increments $\Delta B_k, 1\leq k\leq 2^n$ are the same for $B(\cdot)$ and $B^{(n+1),a}(\cdot)$ by \eqref{trihere}, we have that $X_n(1;\mu^{(n+1)},B)=X_n(1;\mu^{(n+1)},B^{n+1,a})$ and thus,
	\begin{align}\label{4_result}
	&\|X_{n+1}(1)-X_{n+1}^a(1)\|_{\infty}\nonumber \\
	\leq& \|X_{n+1}(1)-X_n(1;\mu^{(n+1)},B)\|_{\infty} + \|X_{n+1}^a(1)-X_n(1;\mu^{(n+1)},B^{n+1,a})\|_{\infty} \nonumber\\
	\leq & \|X_{n+1}(1)-X(1;\mu^{(n+1)},B)\|_{\infty}+\|X_n(1;\mu^{(n+1)},B)-X(1;\mu^{(n+1)},B)\|_{\infty} \nonumber \\
	&+\|X_{n+1}^a(1)-X(1;\mu^{(n+1)},B^{n+1,a})\|_{\infty}+\|X_n(1;\mu^{(n+1)},B^{n+1,a})-X(1;\mu^{(n+1)},B^{n+1,a})\|_{\infty} \nonumber \\
	\leq& 2\Big(\mathcal{P}''(L_1,\|B\|_{\alpha}, \|A\|_{2\alpha},\Gamma_{\widetilde{R}})+\mathcal{P}''(L_1,\|B^{n+1,a}\|_{\alpha}, \|A^{n+1,a}\|_{2\alpha},\Gamma_{\widetilde{R}^{n+1,a}})\Big)\Delta t_n^{2\alpha-\beta}\notag\\
	\end{align}
	The last line follows from \lemref{note lemma4} where quantity $\|B^{n+1,a}\|_{\alpha}, \|A^{n+1,a}\|_{2\alpha},\Gamma_{\widetilde{R}^{n+1,a}}$ is defined for $B^{n+1,a}(\cdot)$ as for $B(\cdot)$ in Definition \ref{alni}. Now, raising inequality \eqref{4_result} to the eighth power and using \lemref{jose1}, we can find polynomial $\mathcal{P}_4(x)>1$  for $x>1$ such that, for all $n$,
	\begin{equation}
	\EEb[\|X_{n+1}(1)-X_{n+1}^a(1)\|_{\infty}^8] \leq \mathcal{P}(L_1)\Delta t_n^{8(2\alpha-\beta)}.
	\end{equation}
\end{proof}

\begin{proof}[Proof of Lemma~\ref{mgt Gaussian}]
	By the Gaussian tail bound  $\int_\xi^{\infty}e^{-\frac{t^2}{2}}dt \leq \frac{1}{\xi}e^{-\frac{\xi^2}{2}}$ for all  $\xi > 0$,
	\begin{align*}
	\mathbb{\mathbb{P}}(M_{\epsilon} > {b})=1-\prod_{n=1}^{\infty}\mathbb{\mathbb{P}}(|{Z_n}|\leq {b}n^{\epsilon})\leq 1-\prod_{n=1}^{\infty}(1-\frac{2}{\sqrt{2\pi}\cdot {b}n^{\epsilon}}{e^{-\frac{{b}^2 n^{2\epsilon}}{2}}}).
	\end{align*}
	Thus, we have $\mathbb E[e^{tM_{\epsilon}}]=\int_0^\infty \mathbb{\mathbb{P}}(e^{tM_{\epsilon}}>{b})db$ which is bounded by
	\begin{align*}
	\int_{0}^\infty \mathbb{\mathbb{P}}(M_{\epsilon}>\frac{\log({b})}{t})d{b} 
	\leq &\int_{{b}_t}^{\infty} (1- \prod_{n=1}^{\infty}(1-\frac{2t}{\sqrt{2\pi}n^{\epsilon}\cdot \log({b})}{e^{-\frac{{\log({b})}^2 n^{2\epsilon}}{2t^2}}}))d{b}< \infty
	\end{align*}
	according to calculation.
\end{proof}

\end{document}